\documentclass[a4paper,11pt]{report}
\usepackage[numbers]{natbib}
\usepackage{graphicx}
\usepackage[margin=3cm]{geometry}
\usepackage[usenames,dvipsnames]{xcolor}
\usepackage{longtable,pifont}
\setlongtables
\usepackage[%
  edge length = 1cm,
  root radius = 0.1cm,
  mark=o
]{dynkin-diagrams}
\usepackage{amsmath,amssymb,amsthm,amscd,mathrsfs}

\usepackage{listings} % Allows verbatim-like environment 'lstlistings' that can break across lines
\DeclareMathOperator{\Lie}{Lie}

\DeclareMathOperator{\ad}{ad}

\DeclareMathOperator{\Der}{Der}
\DeclareMathOperator{\Char}{char}

\usepackage[all,cmtip]{xy}
\usepackage{lscape}
\usepackage{tikz}
\usepackage{enumitem}
\setlist[enumerate,1]{itemsep=3pt,topsep=3pt}

\usepackage{array,booktabs}
\newcolumntype{M}[1]{>{\centering\arraybackslash}m{#1}}

\usepackage{etoolbox}

\usepackage{hyperref}

\definecolor{burntumber}{rgb}{0.54, 0.2, 0.14}
\definecolor{coolblack}{rgb}{0.0, 0.18, 0.39}
\definecolor{darkterracotta}{rgb}{0.8, 0.31, 0.36}
\definecolor{frenchbeige}{rgb}{0.65, 0.48, 0.36}

\makeatletter
\def\@footnotecolor{red}
\define@key{Hyp}{footnotecolor}{%
 \HyColor@HyperrefColor{#1}\@footnotecolor%
}
\patchcmd{\@footnotemark}{\hyper@linkstart{link}}{\hyper@linkstart{footnote}}{}{}
\makeatother
\hypersetup{colorlinks=true,
linkcolor=NavyBlue,
    filecolor=OliveGreen,      
    urlcolor=burntumber,
    citecolor=darkterracotta,
    anchorcolor=frenchbeige,
    footnotecolor=blue}

\newcounter{rownum}
\newcommand{\Ind}{\mathrm{Ind}}

\newtheorem{lemma}{Lemma}[chapter]
\newtheorem{theorem}[lemma]{Theorem}

\newtheorem*{mainproblem}{Main Problem}

\newtheorem{cor}[lemma]{Corollary}
\newtheorem{conjecture}[lemma]{Conjecture}
\newtheorem{prop}[lemma]{Proposition}

\theoremstyle{definition}
\newtheorem{example}[lemma]{Example}

\newtheorem*{algorithm*}{Algorithm}
\newtheorem{defn}[lemma]{Definition}
\theoremstyle{remark}
\newtheorem{remark}[lemma]{Remark}
\newtheorem{remarks}[lemma]{Remarks}

\newtheorem*{acknowledgment}{Acknowledgment}

\newcommand{\Hom}{\mathrm{Hom}}

\newcommand{\opH}{\mathrm{H}}
\newcommand{\D}{\mathscr{D}}
\newcommand{\SO}{\mathrm{SO}}
\newcommand{\Sp}{\mathrm{Sp}}
\newcommand{\SL}{\mathrm{SL}}
\newcommand{\Cl}{\mathrm{Cl}}
\newcommand{\PGL}{\mathrm{PGL}}
\newcommand{\PSp}{\mathrm{PSp}}
\newcommand{\Spin}{\mathrm{Spin}}

\newcommand{\GL}{\mathrm{GL}}

\newcommand{\RR}{\mathscr{R}}

\newcommand{\End}{\mathrm{End}}

\newcommand{\Aut}{\mathrm{Aut}}

\newcommand{\Z}{\mathbb{Z}}

\newcommand{\g}{\mathfrak{g}}

\newcommand{\m}{\mathfrak{m}}

\newcommand{\p}{\mathfrak{p}}

\renewcommand{\l}{\mathfrak{l}}

\newcommand{\h}{\mathfrak{h}}

\newcommand{\Gm}{\mathbb{G}_m}

\newcommand{\C}{\mathbb{C}}

\newcommand{\so}{\mathfrak{so}}
\renewcommand{\sl}{\mathfrak{sl}}

\renewcommand{\sp}{\mathfrak{sp}}
\newcommand{\Ga}{\mathbb{G}_a}

\begin{document}

\title{Complete reducibility and subgroups of exceptional algebraic groups}
\author{Alastair J.\ Litterick, David I.\ Stewart, Adam R.\ Thomas} 

\maketitle

\begin{abstract}
This survey article has two components. The first part gives a gentle introduction to Serre's notion of $G$-complete reducibility, where $G$ is a connected reductive algebraic group defined over an algebraically closed field. The second part concerns consequences of this theory when $G$ is simple of exceptional type, specifically its role in elucidating the subgroup structure of $G$. The latter subject has a history going back about sixty years. We give an overview of what is known, up to the present day. We also take the opportunity to offer several corrections to the literature.
\end{abstract}

\tableofcontents

\begin{acknowledgment}
All three authors wish to thank Martin Liebeck for support and guidance from their PhD studies at Imperial College London through to the present day.

\noindent The second author is supported by Leverhulme grant RPG-2021-080.

\noindent The third author is supported by EPSRC grant EP/W000466/1.
\end{acknowledgment}

\section*{Introduction}

Let $G$ be an affine algebraic group over an algebraically closed field $k$ of characteristic $p\geq 0$. In this article we will only be interested in the case that $G$ is smooth, connected and reductive. Except in one or two places, the subgroups of $G$ encountered in this article will be smooth, and thus we think of $G$ as a variety as per \cite{MR0396773}, rather than adopting the scheme-theoretic language of \cite{milne17}. In any case, we may assume that $G$ is a subgroup of some general linear group defined by the vanishing of some polynomials in the matrix entries---more specifically by a radical ideal. Throughout, all vector spaces, representations and general linear groups will be of finite dimension.

The idea of $G$-complete reducibility is due to J-P.~Serre \cite{Ser3}. It generalises the property of a representation $\rho:H\to G=\GL(V)$ of a group $H$ being completely reducible, by restating the definition in terms of the relationship between the subgroup $\rho(H)\subseteq G$ and the parabolic subgroups of $G$. In this way, the same property can be formulated when $G$ is any connected reductive algebraic group and $H$ is one of its subgroups. 

The notion of $G$-complete reducibility appears in many unexpected places, by virtue of the links it offers between representation theory, group theory, algebraic geometry and geometric invariant theory. For the purposes of this article, it offers the cleanest language to talk about the subgroup structure of $G$. 

Unless otherwise mentioned, we only consider closed subgroups of $G$---those cut out from $G$ by polynomial equations. A linear algebraic group is \emph{unipotent} if it is isomorphic to a subgroup of the upper unitriangular matrices in $\GL(V)$ for some $V$. As the rank of $G$ grows, one sees that it quickly becomes impossible to say very much about the unipotent subgroups of $G$, in much the same way that finite $p$-groups become unmanageable. The same problem arises when $H$ is allowed to have a non-trivial normal connected unipotent subgroup.  Therefore we focus on those subgroups $H$ of $G$ which do not contain a non-trivial normal connected unipotent subgroup; in other words $H$ is \emph{reductive}. Reductive subgroups are also interesting from a geometric perspective, as they are precisely those subgroups $H$ such that the coset space $G/H$ is affine \cite{MR437549}.

Next, a subgroup $H\subseteq\GL(V)$ is the same thing as a faithful representation of $H$. Since we do not want to consider the representation theory of all finite groups, we will assume that $H$ is \emph{connected.} Further still, a connected reductive group $H$ takes the form $H=\D(H)\cdot \RR(H)$ where $\RR(H)$ is a central torus of $H$ and $\D(H)$ is the semisimple derived subgroup of $H$. It would be unenlightening to enumerate all tori in $G$ which commute with $\D(H)$, and so:
\begin{equation}\text{we assume that $H$ is semisimple.}\nonumber \end{equation}

This is still not enough in general to attempt a classification, unless $p=0$. For, in positive characteristic a non-trivial semisimple group has a rich theory of indecomposable modules and it seems hopeless to classify \emph{all} possibilities; therefore one cannot classify all the subgroups of $\GL(V)$ with $\dim V$ unbounded. 

On a more sanguine note: since \emph{irreducible} representations of $H$ are classified by their highest weight, the \emph{completely reducible} representations correspond to lists of such weights. If $p = 0$ then all representations are completely reducible and Weyl's dimension formula gives the dimensions of the irreducible factors, so listing semisimple subgroups of $G=\GL(V)$ reduces to straightforward combinatorics with roots systems and weights. In positive characteristic, understanding subgroups acting completely reducibly means understanding the dimensions of irreducible modules. This is a hard problem. The most recent progress on this problem is due to Williamson and his collaborators---for example, see \cite{RW21}---but regrettably there is no clear description for the dimensions of the irreducible $H$-modules in all characteristics, unless the rank of $H$ is small.

But now suppose that we bound the rank of $G$; more specifically take $G$ to be simple of exceptional type, so that the rank of $G$ is at most $8$. Then there is a realistic prospect of understanding the poset of conjugacy classes of semisimple subgroups of $G$ in all characteristics: any semisimple subgroup $H\subseteq G$ also has rank at most $8$.\begin{mainproblem} Let $G$ be a simple algebraic group of exceptional type. Describe the poset of conjugacy classes of semisimple subgroups of $G$.\end{mainproblem}
Such a project naturally divides along lines prescribed by $G$-complete reducibility, which we explain in \S\ref{sec:strat}. With this in mind, our purpose is twofold:
\begin{enumerate}
  \item We introduce $G$-complete reducibility and the links it provides between representation theory, geometric invariant theory and group theory. We will be light on technical details, but aim to impart the flavour of the techniques and the most important results.
	\item We discuss the current state of affairs in describing semisimple subgroups of the exceptional algebraic groups. We start with a historical overview and then collate the principal results from across the literature. We also correct some errors and omissions that have arisen in this study. Significantly, we update the table in \cite{Stewart21072013}; see Table~\ref{tab:newimrn}.
\end{enumerate}

\subsection*{Prerequisite knowledge}
This is an article about linear algebraic groups over algebraically closed fields and so a healthy knowledge of such groups would be appropriate. Most of the results here do not require a scheme-theoretic background, and for these one of \cite{MR0396773,Spr98,MR1102012} would suffice. A particularly accessible overview of the theory surrounding parabolic subgroups, their Levi factors and associated combinatorics can be found in \cite{MR2850737}. Our use of representation theory will not frequently stray far from the classification of irreducible modules by highest weight, which can be found in these same references. Occasionally a discussion of cohomology takes us into the world of \cite{Jan03}; though we will typically content ourselves with pointing out references to deeper material when appropriate. At points, knowledge of the theory of finite-dimensional complex Lie algebras would be helpful, as covered for example in \cite{MR499562,MR1153249}. 

\subsection*{Notation}
We will mostly introduce relevant notation as needed. Dynkin diagrams and conventions on roots will be as in \cite{Bourb05}. Actions will always be on the left; thus conjugation will be written ${}^gh=ghg^{-1}$. In keeping with this, a group $G$ which decomposes into a semidirect product of a normal subgroup $N$ with a complement $H$ will be written $N\rtimes H\cong G=NH$, since then $(n,h)(m,k)=(n{}^hm,hk)$ maps to $n{}^hmhk$ under this isomorphism.

As mentioned, our algebraic groups are all varieties over algebraically closed fields, and can be thought of as subgroups of an ambient group $\GL(V)$ for some finite-dimensional vector space $V$. The identity component of an algebraic group $G$ is denoted $G^{\circ}$, and $G/G^{\circ}$ is the (finite) component group.

\subsection*{Structure of the paper}

The structure of the paper is as follows. In Part I, \S\S\ref{sec:gcr}--\ref{sec:redalggrp} motivate the theory of $G$-complete reducibility as the natural generalisation of the representation-theoretic notion. This includes a sufficient overview of reductive algebraic groups to state the fundamental definition (Definition~\ref{def:gcr}) and discuss the most useful tools for our applications. In \S\ref{sec:strat} we describe a strategy for classifying subgroups of reductive groups, which arises naturally from the dichotomy between subgroups of $G$ which are $G$-completely reducible, and those which are not.

In Part II we turn to the particular problem of understanding semisimple subgroups of exceptional simple algebraic groups. We begin with a brief historical overview of results in characteristic zero (essentially due to Dynkin) and their translation to positive characteristic (largely due to Seitz and his collaborators), before delving into the current state of the art, and the application of the strategy described in Part I. Finally, in \S\ref{sec:further} we discuss further ongoing research directions in the area.

\addtocounter{chapter}{1}
\chapter*{Part I. \texorpdfstring{$G$}{G}-complete reducibility}

\section{Complete reducibility} \label{sec:gcr}
Almost the first result one encounters in group representation theory is Maschke's Theorem: Given a finite group $H$ and a finite-dimensional $\C H$-module $V$, or equivalently a representation $H \to \GL(V)$, every $H$-submodule (i.e.~$H$-stable subspace) $U \subseteq V$ admits an $H$-stable complement $U'$, so that $V \cong U \oplus U'$ as $H$-modules. The proof proceeds by taking a vector-space direct sum $V = U_0 \oplus U$ and averaging the projection map $\phi: V \to V/U_0 = U$ over $H$ to form the $H$-module homomorphism
\[ V \to U, \qquad v \mapsto \frac{1}{|H|} \sum_{h \in H} h\phi(h^{-1}v) \]
whose kernel is then the required submodule $U'$. An identical proof works over any field $k$ (in fact, over any commutative ring) in which the order $|H|$ is invertible, in particular Maschke's theorem holds whenever $\Char k$ is sufficiently large relative to $|H|$. The analogous result (a consequence of the Peter--Weyl theorem) holds for unitary representations of connected compact topological groups, in particular for compact real Lie groups. Further still, given a complex semisimple Lie group, its Lie algebra is obtained by complexifying the (real) Lie algebra of a compact real Lie group -- and the finite-dimensional representations of all these objects are essentially the same (this is Weyl's unitarian trick \cite[\S 5]{MR1544744}\footnote{For an accessible overview in English, see \cite{Sury}.}). For a direct approach to semisimple complex Lie algebras, one can also follow \cite[C.15]{MR1153249}.

These are some examples of module categories which are \emph{completely reducible} or \emph{semisimple}. For a module of finite dimension over a field, repeated decomposition into submodules expresses it as a direct sum of \emph{irreducible} modules, i.e.~non-zero modules having no proper, non-zero submodules. So when all modules are completely reducible, understanding the category amounts to understanding the irreducible modules. Complete reducibility is certainly not ubiquitous, however. A common example is the following representation of the group of integers under addition: \[\Z \to \GL_2(\C);\quad n \mapsto \begin{pmatrix} 1 & n \\ 0 & 1 \end{pmatrix},\] whose image stabilises a unique $1$-dimensional subspace, spanned by $\binom{1}{0}$, which admits no complementary $\Z$-submodule. Performing a reduction modulo a prime $p$ produces a cyclic group of order $p$ acting on a module over a field of characteristic $p$, which again fails to be completely reducible.\footnote{Note that the condition of Maschke's theorem is violated, since the characteristic $p$ of the field divides $|\Z/p|=p$.}

This formulation of complete reducibility places the focus on the acting group $H$. One could equally decide to fix the target of the representation and ask: \begin{quote}\emph{Given a finite-dimensional vector space $V$, for which subgroups $H$ of $G = \GL(V)$ is $V$ a completely reducible $H$-module?}\end{quote} This is now a question about the subgroup structure of $G$. It may also be more tractable, since $H$ is either trivial or has a faithful module of dimension at most $\dim V$, which puts limitations on $H$. Moreover, this reformulation can be put in purely group-theoretic terms. Recall that a \emph{parabolic subgroup} $P$ of $\GL(V)$ is the stabiliser of a flag of subspaces
\[ V = V_0 \supset V_1 \supset \cdots \supset V_{r} = \{0\}. \]
(See for example \cite[Proposition~12.13]{MR2850737}.) One sees from this that the maximal (proper) parabolic subgroups are stabilisers of (proper, non-zero) subspaces, so a subgroup $H \subseteq \GL(V)$ acts irreducibly if and only if $H$ is contained in no proper parabolic subgroup, and $H$ is completely reducible if and only if: whenever $H$ stabilises a flag of subspaces $V_{i}$ as above, $H$ stabilises an \emph{opposite flag}
\[ V = W_0 \supset W_1 \supset \cdots \supset W_{r} = \{0\}. \]
where $V = V_i \oplus W_{r-i}$ as $H$-modules for all $i$. In other words, whenever $H$ is contained in a parabolic subgroup of $\GL(V)$, it is also contained in an \emph{opposite} parabolic subgroup $P^-$ of $G$. The subgroup of $P$ acting trivially on each quotient $V_{i}/V_{i+1}$ is upper unitriangular according to an appropriate basis, and is therefore a unipotent group. Indeed this turns out to be the largest normal connected unipotent subgroup of $P$, i.e.~its \emph{unipotent radical} $\RR_u(P)$. The intersection $L:=P\cap P^-$ turns out to be a reductive complement to $\RR_u(P)$ in $P$, in other words a \emph{Levi subgroup}. In $\GL(V)$, if $P$ corresponds to a flag as above then Levi subgroups of $P$ are stabilisers of vector-space direct-sum decompositions giving that flag as intermediate sums. That is, a decomposition $V = U_1 \oplus \cdots \oplus U_{r}$, where $V_i = U_{i+1} \oplus U_{i+2} \oplus \cdots \oplus U_{r}$ for each $i$, corresponds to the Levi factor $\GL(U_1) \times \ldots \times \GL(U_r)$. By change of basis, one sees that Levi subgroups of a fixed $P$ are all conjugate (even by elements of $\RR_u(P)$), so that a choice of basis in which the $V_i$ and $W_i$ are spanned by standard basis vectors leads to the following picture:

{\small
\begin{align*}
\renewcommand\arraystretch{1.4}
\underbrace{
\left(\begin{array}{cccc}
  \cline{1-1}
  \multicolumn{1}{|c}{\ast} & \multicolumn{1}{|c}{\ast} & \cdots & \ast \\ \cline{1-2}
  \mathbf{0} & \multicolumn{1}{|c}{\ast} & \multicolumn{1}{|c}{\ddots} & \vdots \\ \cline{2-3}
  \vdots & \ddots & \multicolumn{1}{|c}{\ddots} & \multicolumn{1}{|c}{\ast} \\ \cline{3-4}
  \mathbf{0} & \cdots & \mathbf{0} & \multicolumn{1}{|c|}{\ast} \\ \cline{4-4}
  \end{array}\right)
  }_{\displaystyle P}
  \cong 
    \underbrace{
  \left(\begin{array}{*{4}{c}}
  \cline{1-1}
  \multicolumn{1}{|c}{I_{n_1}} & \multicolumn{1}{|c}{\ast} & \cdots & \ast \\ \cline{1-2}
  \mathbf{0} & \multicolumn{1}{|c}{I_{n_2}} & \multicolumn{1}{|c}{\ddots} & \vdots \\ \cline{2-3}
  \vdots & \ddots & \multicolumn{1}{|c}{\ddots} & \multicolumn{1}{|c}{\ast} \\ \cline{3-4}
  \mathbf{0} & \cdots & \mathbf{0} & \multicolumn{1}{|c|}{I_{n_r}} \\ \cline{4-4}
  \end{array}\right)
  }_{\displaystyle \RR_u(P)}
    \rtimes
    \underbrace{
\left(\begin{array}{*{4}{c}}
  \cline{1-1}
  \multicolumn{1}{|c}{\ast} & \multicolumn{1}{|c}{\mathbf{0}} & \cdots & \mathbf{0} \\ \cline{1-2}
  \mathbf{0} & \multicolumn{1}{|c}{\ast} & \multicolumn{1}{|c}{\ddots} & \vdots \\ \cline{2-3}
  \vdots & \ddots & \multicolumn{1}{|c}{\ddots} & \multicolumn{1}{|c}{\mathbf{0}} \\ \cline{3-4}
  \mathbf{0} & \cdots & \mathbf{0} & \multicolumn{1}{|c|}{\ast} \\ \cline{4-4}
  \end{array}\right)}_{\displaystyle L}
\end{align*}
}

\section{Reductive algebraic groups} \label{sec:redalggrp}

The above discussion generalises at once to other groups with an appropriate notion of parabolic subgroup and Levi subgroup; in particular when $G$ is a reductive algebraic group. The formal definition in this case is that a subgroup $P$ of $G$ is parabolic if $G/P$ is a projective variety. However, a more useful characterisation for us can be given using the structure theory of reductive groups, which we now outline.

Amongst connected linear algebraic groups, the simple objects are by definition those which are non-abelian and have no non-trivial proper connected normal subgroups. Such groups are determined up to isogeny---i.e.~a homomorphism with a finite kernel---by the algebraically closed field $k$ and one of the Dynkin diagrams below, which are divided into those of \emph{classical type} $A_n$ $(n \ge 1)$, $B_n$ ($n\geq 2$), $C_n$ ($n\geq 3$), $D_n$ ($n\geq 4$), and those of \emph{exceptional type} $E_6$, $E_7$, $E_8$, $F_4$ or $G_2$.

\begin{center} \label{dynkins}
\begin{tabular}{cc}
$A_n$ & \dynkin[%
      labels={1,2,3,n-1,n},
      text style/.style={scale=1},
      label macro/.code={\drlap{#1}},
      label directions={,,,below left,}]A{ooo.oo} \\ { } \\
$B_n$ & \dynkin[%
      labels={1,2,{},n-1,n},
      text style/.style={scale=1},
      label macro/.code={\drlap{#1}},
      label directions={,,,below left,}]B{oo.ooo} \\ { } \\
$C_n$ & \dynkin[labels={1,2,{},n-1,n},text style/.style={scale=1},
      label macro/.code={\drlap{#1}},
      label directions={,,,below left,}]C{oo.ooo} \\ { } \\
$D_n$ & \hspace{1.2em} \dynkin[%
      labels={1,2,{},n-2,n-1,n},
      text style/.style={scale=1},
      label macro/.code={\drlap{#1}},
      label directions={,,,right,,}]D{} \\
\shortstack{$E_n$ \\ $(n = 6,7,8)$} & \dynkin[label,labels={1,2,3,4,5,n},text style/.style={scale=1},
      label macro/.code={\drlap{#1}}]E{ooooo.o} \\
$F_4$ & \dynkin[label,text style/.style={scale=1},
      label macro/.code={\drlap{#1}}]F4 \\
$G_2$ & \dynkin[%
      label,
      reverse arrows,
      text style/.style={scale=1},
      label macro/.code={\drlap{#1}}]G2 \\
\end{tabular}
\end{center}
We also need to mention the additive group of the field $k$, often denoted $\Ga$, and the multiplicative group $k^{\ast}$, denoted $\Gm$. Since our groups are varieties over an algebraically closed field, it follows that groups with a subnormal series whose successive quotients are all $\Ga$ coincide with connected unipotent groups. Similarly, linear algebraic groups with a composition series whose quotients are isomorphic to $\Gm$ are precisely direct products $\Gm^r$ for some $r$, and are called \emph{tori}. It is then a theorem that a a connected soluble linear algebraic group is the semidirect product of a connected unipotent group and a torus.

Tori play an important role in the structure theory of reductive groups. All maximal tori in a linear algebraic group are conjugate to one another, and their dimension is called the \emph{rank} of the group. This is the number $n$ of nodes in the Dynkin diagram when $G$ is simple. For a maximal torus $T$ of $G$, we let $W(G) = N_G(T)/T$ be the Weyl group of $G$ with respect to $T$.

The simply-connected groups of type $A_n$--$D_n$ are respectively $\SL_{n+1}$, $\Spin_{2n+1}$, $\Sp_{2n}$ and $\Spin_{2n}$, and from these one obtains the others as quotients with finite kernels. For our purposes there is often no harm in working with $\SL_{n+1}$, $\SO_n$ and $\Sp_{2n}$ which have more accessible descriptions in terms of the natural module and its quadratic or symplectic form.

A linear algebraic group $G$ contains a unique maximal connected normal soluble subgroup $\RR(G)$, the \emph{radical} of $G$, containing the unique maximal connected normal unipotent subgroup $\RR_u(G)$ of $G$, the \emph{unipotent radical}. Then $G$ is called:
\begin{itemize}
  \item[$\cdot$] \emph{reductive} if $\RR_u(G)$ is trivial; this is equivalent to $G^\circ$ being an almost-direct product\footnote{commuting product of normal subgroups, with pairwise finite intersections} of simple algebraic groups and a torus.
  \item[$\cdot$] \emph{semisimple} if $G$ is connected and $\RR(G)$ is trivial; this is equivalent to $G$ being an almost-direct product of simple algebraic groups.
\end{itemize}
For any linear algebraic group $G$, the quotient $G/\RR_u(G)$ is reductive and if $G$ is connected then $G/\RR(G)$ is semisimple.

From now on let $G$ be connected. We are interested in finite-dimensional \emph{rational} representations of $G$, which can be identified with homomorphisms $G \to \GL(V)$ of algebraic groups. One easy case to describe is when $G=\Gm$. The irreducible representations of $\Gm$ are $1$-dimensional; if $V=\langle v\rangle$ is one such, then $x\cdot v=x^rv$ for some $r\in\Z$ and so corresponds with a homomorphism $x\mapsto x^r$ of $\Gm$ to $\Gm$. Even better, $\Gm$ always acts completely reducibly, so that a representation up to isomorphism is identified with a list of integers. More generally if $G=T$ is a torus $T\cong \Gm\times\dots\times\Gm$, then an irreducible representation $V$ is still $1$-dimensional, determined up to isomorphism by the action \[(x_1,\dots,x_s)\cdot v=x_1^{r_1}\cdots x_s^{r_s}v,\] which identifies with an element $\lambda\in \Hom(T,\Gm)=:X(T)$; then $\lambda$ is called a \emph{weight} of $T$. (We often identify $\lambda$ with the corresponding $1$-dimensional representation.) Since $T$ acts completely reducibly, a representation for $T$ is simply a list of its weights. 

Taking inspiration from the theory of Lie groups, one can construct a Lie algebra $\g=\Lie(G)$ from $G$ which affords a representation of $G$ through an \emph{adjoint action}. Since maximal tori in $G$ are conjugate, the collection of weights of a maximal torus $T$ on $\Lie(G)$ does not depend on the choice of $T$. The zero weight-space is $\Lie(T)$, and when $G$ is reductive, the non-zero weights are called \emph{roots} and form the \emph{root system} which is denoted $\Phi = \Phi(G,T)$. Also by the conjugacy of maximal tori, the isomorphism type of the Weyl group $W(G)$ is independent of $T$, and $W(G)$ acts naturally on $X(T)$ and $\Phi(G,T)$.

A reductive group $G$ has a maximal connected soluble subgroup $B$ called a \emph{Borel subgroup} which by Borel's fixed point theorem is unique up to conjugacy in $G$. This decomposes as $B=US$, where $U = \RR_u(B)$ is a maximal connected unipotent subgroup of $G$, and $S$ is a maximal torus of $G$. When $S=T\subseteq B$, the set of roots $\Phi(B,T)$ contains exactly half the elements of $\Phi(G,T)$ and defines a positive system $\Phi^+\subset \Phi$. Moreover, $\Phi$ has a \emph{base of simple roots} $\Delta$: a minimal subset of $\Phi^+$ from which all elements of $\Phi^+$ are obtained as non-negative integer sums. We call $|\Delta|$ the \emph{semisimple rank} of $G$. One can define a scalar product (the Killing form) on the $\mathbb{R}$-linear span of the roots; the Dynkin diagram of $G$ then encodes the resulting lengths and relative angles of the simple roots.

Within $X(T)$ we single out the \emph{dominant weights}, $X(T)^{+}$, which are those having non-negative inner product with each positive root. If $\lambda\in X(T)^+$ then we identify $\lambda$ with the corresponding $1$-dimensional $T$-module. Composing with the map $B\to B/U\cong T$, we get a $B$-module $\lambda$. One can define an induced module $\opH^0(\lambda):=\Ind_B^G(\lambda)$ in the category of rational $G$-modules. Since $G/B$ is projective, this module has finite dimension, and can be viewed as a reduction modulo $p$ of the irreducible module $L_{\mathbb{C}}(\lambda)$ for the complex Lie algebra $\g_{\mathbb{C}}$, so the weights of $\Ind_B^G(\lambda)$ are given by Weyl's character formula. If we define $L(\lambda)$ to be the socle of $\Ind_B^G(\lambda)$ then it turns out to be simple, with $\lambda$ as its highest weight. These turn out to be all the irreducible modules: \cite[II.2.4]{Jan03}.\footnote{Another reduction modulo $p$ gives the Weyl module $V(\lambda)$, which has $L(\lambda)$ as its head.}

Parabolic subgroups $P$ of $G$ can now be characterised as follows: $P$ is any subgroup containing a Borel subgroup. Fixing $B\subseteq P$ it turns out that $\Phi(P,T)$ is an enlargement of $\Phi(B,T)$ obtained by a suitable choice of simple roots $\Delta'\subseteq\Delta$ and taking the smallest additively-closed subset of $\Phi(G,T)$ containing $\Phi(B,T)$ and $-\Delta'$. If $r$ is the semisimple rank of $G$, then there are $2^r$ non-conjugate parabolic subgroups containing a given Borel subgroup $B$; these correspond to the possible subsets of nodes in the Dynkin diagram. The remaining parabolic subgroups are all conjugate to one of these under the action of $G$.\label{parabsec}

Just as we saw in $\GL(V)$, a parabolic subgroup admits a \emph{Levi decomposition} $P = \RR_u(P)\rtimes L$, where $L$ is a reductive group called a \emph{Levi subgroup} of $P$; all such Levi subgroups are conjugate by elements of $\RR_u(P)$. If one insists that $T\subseteq L$ then $L$ is unique, and the Levi decomposition can be seen at the level of roots. If $P$ corresponds to $\Delta'$ then the roots arising as sums of roots contained in $\Delta' \cup-\Delta'$ give those of $\Phi(L,T)$; their complement in $\Phi(P,T)$ are the roots in $\RR_u(P)$. 

\begin{example}
Let $G$ be simple of type $F_4$. Elementary root system combinatorics (see e.g.~\cite[\S 2.10]{MR1066460}) tell us that $G$ has $48$ roots (hence $24$ positive roots). If we pick the two middle nodes of the Dynkin diagram of $G$ (cf.~page \pageref{dynkins}), the corresponding roots and their negatives generate a subsystem of type $B_2$, which has $8$ roots. So in the corresponding parabolic subgroup $P = \RR_u(P)L$ of $G$, the derived subgroup of $L$ will be simple of type $B_2$, which has dimension $10$ ($8$ roots, plus a $2$-dimensional maximal torus). The connected centre of $L$ will be a $2$-dimensional torus since $L$ contains a maximal torus of $G$, which has rank $4$. Finally, counting positive roots in $G$ and $L$, the unipotent radical $\RR_u(P)$ has dimension $24 - 4 = 20$.

In fact, much of the structure of $\RR_u(P)$ as an $L$-group can also be quickly deduced from the root system. We will return to this in \S\ref{ss:cohom}.
\end{example}

The following result \cite[Theorem~2.5]{MR0294349} is fundamental in our study and will henceforth be called the \emph{Borel--Tits theorem}.
\begin{theorem}[Borel--Tits]\label{thm:bt}
Let $G$ be a connected reductive algebraic group, and let $X$ be a unipotent subgroup of $G$. Then there exists a (canonically-defined) parabolic subgroup $P$ of $G$ such that $X \subseteq \RR_u(P)$.
\end{theorem}

This implies in particular that a maximal subgroup of a reductive group $G$ is either reductive or parabolic. And a maximal connected subgroup of a semisimple group $G$ is either semisimple or parabolic.

\section{\texorpdfstring{$G$}{G}-complete reducibility}

At last, we come to the central definition.
\begin{defn}[\protect{\cite[p.~19]{Ser3}},\ {\cite[\S 3.2.1]{MR2167207}}] \label{def:gcr}
Let $G$ be a connected reductive algebraic group over the algebraically closed field $k$.

A subgroup $H$ of $G$ is called \emph{$G$-completely reducible} ($G$-cr) if, whenever $H$ is contained in a parabolic subgroup $P$ of $G$, there exists a Levi subgroup $L$ of $P$ with $H \subseteq L$. Similarly, $H$ is called \emph{$G$-irreducible} ($G$-irr) if $H$ is contained in no proper parabolic subgroup of $G$; it is \emph{$G$-reducible} if it is not $G$-irr. Lastly $H$ is \emph{$G$-indecomposable} if $H$ is in no proper Levi subgroup of $G$ and \emph{$G$-decomposable} otherwise.
\end{defn}

\begin{remarks} \leavevmode
\begin{enumerate}
  \item An abstract subgroup $H$ of $G$ and its Zariski closure are contained in precisely the same closed subgroups of $G$, in particular, the same parabolic subgroups and Levi subgroups thereof. Therefore $H$ is $G$-cr if and only if its Zariski closure is, so it does no harm to work only with closed subgroups.
  \item For disconnected groups $G$, one can define \emph{R-parabolic} and \emph{R-Levi} subgroups of $G$ in terms of limits of certain morphisms. These coincide with parabolic and Levi subgroups when $G$ is connected, hence $G$-complete reducibility and many results here generalise to disconnected groups. A full discussion is beyond the scope of this article; we direct the reader to \cite[\S 6]{MR2178661} for a thorough overview.
\end{enumerate}
\end{remarks}

Many general results in the representation theory of groups can be viewed as cases of statements about $G$-complete reducibility, when $G$ is specialised to $\GL(V)$. We now collect some of these.

\subsection{Characteristic criteria for complete reducibility} \label{ss:criteria}
Jantzen proved in \cite{MR1635685} that if $H$ is connected reductive and $V$ is an $H$-module with $\dim V \leq p$ then $V$ is completely reducible; this bound was improved by McNinch in \cite{MR1476899}. A more general statement is:
\begin{theorem} \label{thm:gcr-red}
Let $G$ be a connected reductive algebraic group over $k$ of characteristic $p\geq 0$ and let $H$ be a connected subgroup of $G$. If $H$ is $G$-cr then $H$ is reductive. Conversely, if $H$ is reductive and either $p=0$ or $p$ is greater than the ranks of all simple factors of $G$, then $H$ is $G$-cr.
\end{theorem}

The implication `$G$-cr $\Rightarrow$ reductive' follows from the Borel--Tits theorem: one can construct a parabolic subgroup $P$ containing $H$ such that $\RR_u(H)\subseteq\RR_u(P)$. Thus if $\RR_u(H)$ is non-trivial then $H$ is in no Levi subgroup of $P$.

The `reductive $\Rightarrow$ $G$-cr' direction can be deduced from Jantzen's or McNinch's result if $G$ is simple and classical. The exceptional case, tackled in \cite{MR1329942}, relies on a careful study of the $1$-cohomology arising from the conjugation action of subgroups of parabolics on unipotent radicals. We discuss this technique in more detail in \S\ref{s:nonGcr}, where we explain how one can find non-$G$-cr subgroups, when they exist.

Given the result for $G$ simple, the general case follows in short order, since parabolic subgroups and Levi subgroups in $G$ are commuting products of the central torus $Z(G)^{\circ}$ with parabolic and Levi subgroups of the simple factors of $G$.

One of Serre's initial motivations for studying $G$-complete reducibility was in studying complete reducibility of representations. Specifically, the starting point for Serre's lectures \cite{Ser3} is the observation of Chevalley \cite{MR0051242} that in characteristic $0$, tensor products of completely reducible modules for an \emph{arbitrary} group are again completely reducible. This fails in positive characteristic, however Serre observed \cite{MR1253203} that it does hold when the characteristic is large relative to the dimension of the modules in question, and he then began asking questions of the form: If a tensor product (or symmetric power, or alternating power) of modules is completely reducible, must the initial module(s) also be completely reducible? Or conversely? The answer \cite{MR1467165} depends naturally on certain congruence conditions on the characteristic. See \S\ref{s:hereditary} for more on this. 

In a related vein, in \cite[\S 5.2]{MR2167207}, for a reductive group $G$ with a finite-dimensional $G$-module $V$, Serre defines an invariant $n(V)$ in terms of the weights of $G$ on $V$; he then proves that if the characteristic is larger than $n(V)$ and the identity component of the kernel of $G \to \GL(V)$ is a torus, then $H \subseteq G$ is $G$-cr if and only if $V$ is a completely reducible $H$-module \cite[Theorem~5.4]{MR2167207}. Thus in sufficiently large characteristic, $G$-complete reducibility can indeed be detected on the level of $G$-modules. If $V = \Lie(G)$ is the adjoint module then $n(V) = 2h_G - 2$, where $h_G$ is the Coxeter number of $G$. Note that this bound is typically much larger than that given in Theorem~\ref{thm:gcr-red}.

\subsection{Equivalence with strong reductivity} \label{sec:BMRresults}

A major result in complete reducibility generalises the idea that a completely reducible module is the direct sum of a \emph{unique} list of simple ones. In place of the simple summands of a module, the focus is on the Levi subgroup stabilising the decomposition.
\begin{theorem}[{\cite[Corollary~3.5]{MR2178661}}] \label{thm:gcr-strong-red}
  For a subgroup $H$ of a connected reductive algebraic group $G$, the following are equivalent.
  \begin{enumerate}[label=\normalfont(\roman*)]
    \item $H$ is $G$-cr; \label{bmr-i}
    \item $H$ is $C_G(S)$-irr for some maximal torus $S$ of $C_G(H)$; \label{bmr-ii}
    \item for every parabolic subgroup $P$ of $G$ which is minimal with respect to containing $H$,
  the subgroup $H$ is $L$-irr for some Levi subgroup $L$ of $P$;
    \item there exists a parabolic subgroup $P$ of $G$ which is minimal with respect to containing
  $H$, such that $H$ is $L$-irr for some Levi subgroup $L$ of $P$.
  \end{enumerate}
\end{theorem}
In the prototype setting $G = \GL(V)$, if $V = V_1 \oplus \cdots \oplus V_r$ is a decomposition of $V$ into irreducible $H$-modules then the torus $S$ in \ref{bmr-ii} consists of all elements inducing scalars on each summands $V_i$, and $C_G(S)$ is the direct product of the subgroups $\GL(V_i)$.

\begin{remark} \label{rem:minLevis}
Levi subgroups of $G$ are precisely the centralisers of sub-tori of $G$, and moreover all maximal tori in any linear algebraic group are conjugates of one another. In particular all maximal tori of $C_{G}(H)$ are $C_{G}(H)$-conjugate, and it follows that their centralisers, which are those Levi subgroups of $G$ that are minimal subject to containing $H$, are also conjugate to one another by elements of $C_G(H)$, and the ranks of their centres equal the rank of $C_G(H)$. For the same reason, if $H$ is $C_G(S)$-irr for some maximal torus $S$ of $C_G(H)$, it is in fact $C_G(S)$-irr for \emph{all} such $S$.
\end{remark}

Subgroups satisfying \ref{bmr-ii} in Theorem~\ref{thm:gcr-strong-red} were termed \emph{strongly reductive} by Richardson \cite[Def.~16.1]{MR952224}. Richardson studied strongly reductive subgroups from a geometric viewpoint, centred around the following result.
\begin{theorem}[{\cite[\S 16]{MR952224}}]
Suppose that $H$ is the Zariski closure of a finitely-generated subgroup $\left<h_1,\ldots,h_n\right>$ of $G$. Then $H$ is $G$-cr (resp.~$G$-irr) if and only if the orbit $G \cdot (h_1,\ldots,h_n)$ is Zariski closed in $G^n$ (resp.~a \emph{stable point} of $G^n$).
\end{theorem}
\begin{remarks} \leavevmode
\begin{enumerate}
  \item This theorem is a lynchpin of the results in \cite{MR2178661} and subsequent work. It relies on some geometry and geometric invariant theory which we omit; we direct the reader to \cite{MR952224} and \cite{MR2178661}.
  \item In geometric invariant theory, a \emph{stable point} of a $G$-variety is a point whose $G$-orbit is closed and whose stabiliser is a finite extension of the kernel of the action. Thus when $G$ acts on $G^n$ by simultaneous conjugation, a stable point is one whose orbit is closed and whose centraliser is a finite extension of the centre $Z(G)$.
  \item One can drop the hypothesis that $H$ is the closure of a finitely generated subgroup---one need only pick a sufficiently large $n$-tuple so that $H$ and the elements of the $n$-tuple generate the same associative subalgebra of $\End(V)$ for some faithful representation $G \to \GL(V)$; this is always possible for dimension reasons. Such an $n$-tuple is called a \emph{generic tuple} \cite[Def.~2.5]{MR2860266}.
\end{enumerate}\end{remarks}

A highlight of what can be proved through this approach is the following theorem.

\begin{theorem}[{\cite{MR2178661}}] \label{thm:BMRClifford}
  Let $G$ be a connected reductive algebraic group and $H$ be a $G$-cr subgroup of $G$. 
  \begin{enumerate}[label=\normalfont(\roman*)]
    \item If $N \triangleleft H$ then $N$ is $G$-cr. \label{gcr-props-i}
    \item The subgroups $N_G(H)$ and $C_G(H)$ are $G$-cr. More generally, if $K$ is a subgroup of $G$ with 
    \[ H C_G(H)^{\circ} \subseteq K \subseteq N_G(H), \]
    then $K$ is $G$-cr. \label{gcr-props-ii}
  \end{enumerate}

In particular, a subgroup $X \subseteq G$ is $G$-cr if and only if $N_G(X)$ is $G$-cr.
\end{theorem}

Part \ref{gcr-props-i} generalises Clifford's theorem in representation theory: a completely reducible module for a group remains completely reducible upon restriction to a normal subgroup. Part \ref{gcr-props-ii} gives a converse to this, and inspires the following question: \begin{quote}\emph{If $H$ is a commuting product $H = AB$, where $A$ and $B$ are $G$-cr, must $H$ also be $G$-cr?}\end{quote}By \ref{gcr-props-ii} the answer is yes if $A = C_G(B)$ and $B = C_G(A)$. The answer is also positive when the characteristic is large enough:
\begin{theorem}[{\cite[Theorem~1.3]{MR2431255}}] \label{thm:commuting}
Let $G$ be a connected reductive algebraic group in characteristic $p \ge 0$ and let $A$ and $B$ be commuting $G$-cr subgroups of $G$ such that either $p = 0$; or $p > 3$; or $p = 3$ and $G$ has no simple factors of exceptional type. Then $AB$ is also $G$-cr.
\end{theorem}

The proof in \cite{MR2431255} involves some case-by-case arguments depending on the Lie type of $G$, although a uniform argument is also possible if one is willing to relax the bound on the characteristic \cite[Proposition.~40]{MR2309195}.

\begin{remark}
To date, the authors are not aware of a reductive group $G$ and a pair of commuting $G$-cr subgroups $A$ and $B$ (connected or otherwise) whose product $AB$ is non-$G$-cr when $p = 3$. It is therefore plausible that the above theorem holds whenever $p \neq 2$. Such examples do however exist when $p = 2$, cf.~\cite[Ex.~5.3]{MR2431255}.
\end{remark}

One powerful feature of $G$-complete reducibility is that one can allow the group $G$ to change under various constructions. For instance, in \cite{MR2178661} it is shown that if $G = G_1 \times G_2$ is a direct product of reductive groups then $H \subseteq G$ is $G$-cr or $G$-irr if and only if both images under projection to a factor $G_i$ are $G_i$-cr (resp.~$G_i$-irr). Similarly, taking a quotient by a normal subgroup $N$ sends $G$-cr subgroups to $(G/N)$-cr subgroups, and the converse also holds if $N^{\circ}$ is a torus.

The following result is useful in what follows. Recall that a linear algebraic group $S$ is called \emph{linearly reductive} if all its rational representations are completely reducible. In characteristic $0$ this is equivalent to $S$ being reductive, whereas in positive characteristic this means that $S^\circ$ is a torus and the order of the finite group $|S/S^\circ|$ is coprime to $p$.
\begin{theorem}[{\cite[Lemma~2.6, Corollary~3.21]{MR2178661}}] \label{thm:BMRcentraliser}
Let $S$ be a linearly reductive subgroup of a connected reductive algebraic group $G$. Then $S$ is $G$-cr, and if $H = C_G(S)^\circ$ then a subgroup of $H$ is $H$-cr if and only if it is $G$-cr.
\end{theorem}
In the particular case that $S$ is a torus, $C_G(S)$ is a Levi subgroup of $G$ and the result in this case was first proved in \cite[Proposition~3.2]{MR2167207}; this forms part of the proof of Theorem~\ref{thm:gcr-strong-red}. The following corollary justifies our focus on semisimple subgroups:
\begin{cor} \label{cor:linRedCent}
Let $G$ be connected reductive, $H$ a subgroup of $G$ and $S$ a subgroup of $C_G(H)$. Then $H$ is $G$-cr if and only if $HS$ is $G$-cr. In particular, a connected reductive subgroup $H$ of $G$ is $G$-cr if and only if its (semisimple) derived subgroup $\D(H)$ is $G$-cr.
\end{cor}

\begin{proof}
In characteristic $0$, the subgroup $H$ is $G$-cr if and only if it is reductive (Theorem \ref{thm:gcr-red}), which is the case if and only if $HS$ is reductive.

In positive characteristic, using Theorem~\ref{thm:BMRcentraliser} we can replace $G$ with $C_G(S)^\circ$ so that $S$ is central in $G$. Since $S^\circ$ is a torus, by the above discussion the subgroups $H$ and $HS$ are each $G$-cr if and only if $HS/S$ is $(G/S)$-cr.
\end{proof}

\subsection{Separability, reductive pairs}
A subgroup $H$ of $G$ is called \emph{separable} if the Lie algebra $\Lie(C_{G}(H))$ coincides with the fixed-point space $C_{\Lie(G)}(H)$; the latter always contains the former but can in general be larger.\footnote{In scheme-theoretic language, this is equivalent to the centraliser $C_G(H)$ being a smooth subgroup scheme of $G$.} If $\mathbf{h} \in G^n$ is a topological generating tuple (or generic tuple) for $H$, and $G$ acts on $G^n$ by simultaneous conjugation, the statement is also equivalent to saying the orbit map $G \to G \cdot \mathbf{h}$ is a separable morphism \cite[Proposition~6.7]{MR1102012}, whence the terminology. In $\GL(V)$, all closed subgroups are separable, cf.\ \cite[Lemma 3.5]{MR3042602}. For a general reductive group $G$, non-separable subgroups are always a low-characteristic phenomenon. Indeed, the main result of \cite{MR3042602} shows that the statement ``all closed subgroups are separable'' holds if and only if the characteristic is $0$ or \emph{pretty good} for $G$. The latter is a very mild condition - see \cite[Definition~2.11]{MR3042602} for the precise definition. For instance if $G$ is simple of exceptional type then a prime $p$ is pretty good for $G$ unless $p = 2$ or $3$, or $G = E_8$ and $p = 5$.

\label{par:reductive-pair}Next, a pair $(G,H)$ of reductive groups with $H \subseteq G$ is called a \emph{reductive pair} \cite{MR217079} if $\Lie(H)$ is an $H$-module direct summand of $\Lie(G)$.\footnote{While \cite{MR217079} introduces this only for connected groups, the concept and results apply equally well for disconnected subgroups.} Again, $(G,H)$ is always a reductive pair if the underlying characteristic is large relative to $G$ (for instance $\Lie(G)$ is a completely reducible module for all reductive subgroups in sufficiently large characteristic).

The application of these two concepts to complete reducibility is now as follows.
\begin{theorem}[{\cite[Theorem~3.35, Corollary~3.36]{MR2178661}}] \label{thm:sep-red}
Let $(G, M)$ be a reductive pair and let $H$ be a separable subgroup of $G$ which is contained in $M$. If $H$ is $G$-cr then $H$ is $M$-cr.

In particular, if $(\GL(V),M)$ is a reductive pair and $H \subseteq M$ acts completely reducibly on $V$, then $H$ is $M$-cr.
\end{theorem}
The proof in \emph{op.~cit.}~is geometric, following \cite{MR217079}. If $H$ corresponds to $(h_1,\ldots,h_n) \in G^n$, that is, if this is a generic tuple for $H$ or if $H$ is the closure of $\left<h_1,\ldots,h_n\right>$, then under the given hypotheses it is shown that the $G$-orbit $\mathcal{O} = G \cdot \mathbf{h}$ under simultaneous conjugacy splits into finitely many Zariski-closed $M$-orbits in $\mathcal{O} \cap M^n$. So if $\mathcal{O}$ is closed in $G^n$, then the $M$-orbits on $\mathcal{O} \cap M^n$ are closed in $M^n$, which in turn implies that $H$ is $M$-cr.

\subsection{\texorpdfstring{$G$}{G}-complete reducibility in classical groups}

In this section, we write $G = \Cl(V)$ to mean that $G$ is one of the groups $\SL(V)$, $\Sp(V)$ or $\SO(V)$, where $V$ carries either the zero form or a non-degenerate alternating or quadratic form, respectively. Then parabolic subgroups of $G$ are the stabilisers of flags of subspaces where the relevant form vanishes, i.e.~\emph{totally isotropic} and \emph{totally singular} subspaces (respectively) when the form is non-degenerate, cf.\ \cite[Proposition~12.13]{MR2850737}. When the form is non-degenerate, two parabolic subgroups corresponding to flags $(V_{i})_{i = 1,\ldots,r}$ and $(W_{i})_{i = 1,\ldots,s}$ are opposite if $r = s$ and $V$ is an orthogonal direct sum $V_{i} \perp W_{i}^{\perp}$ for each $i$, where $W_{i}^{\perp}$ is the annihilator of $W_i$ relative to the alternating or quadratic form on $V$. Thus sufficient understanding of the action of $H$ on $V$ tells us whether or not $H$ is $G$-cr.

When $G = \SL(V)$ we have seen that a subgroup of $G$ is $G$-cr if and only if it is completely reducible on $V$. By Theorem~\ref{thm:BMRcentraliser} this also holds for classical groups $G = \Cl(V)$ in characteristic not $2$, since $G$ is then the centraliser in $\SL(V)$ of an involutory outer automorphism.\footnote{As an alternative proof, when $p \neq 2$ one can show that $(\GL(V),G)$ is a reductive pair; and since every subgroup is separable in $\GL(V)$, Theorem~\ref{thm:sep-red} tells us that every $\GL(V)$-cr subgroup of $G$ is $G$-cr.} In this case, classifying all $G$-cr semisimple subgroups of $G = \Cl(V)$ amounts to understanding the dimensions and Frobenius--Schur indicators of all irreducible modules of dimension at most $\dim V$, for all semisimple groups. 

The condition $\Char k = p \neq 2$ is necessary to make these assertions. If $p \neq 2$ then recall that a quadratic form $q$ gives rise to a symmetric bilinear form $B$ on $V$ via
\[B(v,w):=\frac{1}{2}\left(q(v+w)-q(v)-q(w)\right),\]
and $q$ can be recovered from $B$ via $q(x) = B(x,x)$. If $p = 2$ then $B(v,w):=q(v+w)-q(v)-q(w)$ defines a symmetric bilinear form, but $B(x,x) = 0$ so $q$ can no longer be recovered. If $B$ is non-degenerate and $B(v,w)\neq 0$ then $B$ is also non-degenerate on $\langle v,w\rangle^\perp$ and $V$ must have even dimension. Thus the bilinear form on the natural module for $\SO_{2n+1}$ has a $1$-dimensional radical, spanned by a non-singular vector for $q$. For convenience, define the \emph{natural irreducible module} for $\Cl(V)$ to be the largest non-trivial irreducible quotient of $V$; this is $V$ itself unless $p=2$ and $G = \SO(V) \cong \SO_{2n+1}$, in which case it is the $2n$-dimensional quotient by the radical of the bilinear form.

The following example is attributed to M.~Liebeck in \cite[Example 3.45]{MR2178661}.
\begin{example} \label{eg:Dn-cr}
Let $\Char k = 2$ and let $H$ be a group preserving a symplectic (resp.\ orthogonal) form on an irreducible module $W$. Let $V = W \perp W$ be an orthogonal direct sum. Then $V$ has a unique non-zero totally isotropic (resp.\ totally singular) $H$-submodule. Thus $H$ is $\GL(V)$-cr but not $\Cl(V)$-cr.
\end{example}
\proof By Schur's lemma, each proper non-zero $H$-submodule of $V$ is the image of a diagonal embedding $W\to V$ by $w\mapsto (aw,bw)$ for $[a:b]\in\mathbb{P}^1(k)$. If $B$ is an $H$-invariant non-degenerate alternating form on $W$ then $B((aw,bw),(au,bu))=(a^2+b^2)B(w,u)=(a+b)^2B(w,u)$. This is zero if and only if $a=b$. The orthogonal case is similar.\qed 

\subsection{\texorpdfstring{$G$}{G}-irreducibility in classical groups}
By the characterisation of parabolic subgroups above, a subgroup of $G = \Cl(V)$ is $G$-irr if and only if it preserves no proper non-zero totally isotropic (or totally singular) subspace. In more detail: 

\begin{prop} \label{prop:irredclass}
Let $G$ be a simple algebraic group of classical type with $V$ the natural irreducible $G$-module, and let $H$ be a subgroup of $G$. Then $H$ is $G$-irr if and only if one of the following holds:
\begin{enumerate}[label=\normalfont(\roman*)]
\item $G$ has type $A_n$ and $V$ is an irreducible $H$-module;
\item $G$ has type $B_n$, $C_n$ or $D_n$ and $V = V_1 \perp \ldots \perp V_k$ as $H$-modules, where the $V_i$ are non-degenerate, irreducible and pairwise inequivalent; \label{prop:irredclass-ii} 
\item $p=2$, $G$ has type $D_{n}$ and $H$ fixes a non-singular vector $v \in V$, such that $H$ is $G_v$-irr in the point stabiliser $G_v$ and does not lie in a subgroup of $G_{v}$ of type $D_{n-1}$. \label{prop:irredclass-iii}
\end{enumerate}
\end{prop}
(See \cite[Proposition~3.1]{Litterick2018a} for a full proof.)

\begin{remarks} \leavevmode
\begin{enumerate}[label=(\alph*)]
\item In case \ref{prop:irredclass-iii}, the bilinear form preserved by $G\cong\SO_{2n}$  is alternating when $\Char k = 2$, hence every $1$-space is isotropic. Now $G_v$ preserves the space $\left< v \right>^\perp$ of codimension $1$ in $V$, and as $v$ is nonsingular, $q$ is non-degenerate on this space, so $G_v\cong\SO_{2n-1}$ is simple of type $B_{n-1}$. For more details, see for instance \cite[Proposition~4.1.7]{MR1057341}.
\item This proposition can be applied to any simple group of type $A$--$D$, such as $\PSp_{2n}$ or the half-spin groups of type $D_n$. For if $G$ is simple and classical then $G$ is related to some $\Cl(V)$ by a quotient or an extension by a finite central subgroup---or if $G$ is a half-spin group, then one of each\footnote{In characteristic $2$ or if $G$ has type $A_{n}$ with $p \mid n + 1$, one must work with central subgroup \emph{schemes}, since the possibilities for $G$ can be isomorphic as abstract groups. However, this does not impact our discussion here.}. If $Z$ is such a finite central subgroup then as $G$ is connected and reductive, $Z$ is contained in all maximal tori, hence in all parabolic subgroups. It follows that a subgroup $H$ of $G$ is $G$-irr if and only if its preimage $HZ$ or quotient $HZ/Z$ is $G$-irr. Furthermore if $\Char k = 2$ then taking the quotient by the $1$-dimensional radical of the bilinear form induces an \emph{exceptional isogeny} $\psi: \SO_{2n+1}\to\Sp_{2n}$ (more details on p.~\pageref{bn_cn}). In this case, $\psi$ is bijective and it follows that $H\subseteq \SO_{2n+1}$ is $\SO_{2n+1}$-cr if and only if $\psi(H)$ is $\Sp_{2n}$-cr.
\end{enumerate}
\end{remarks}

\section{A strategy for classifying semisimple subgroups} \label{sec:strat}
We remind the reader that we wish to tackle the following:

\begin{mainproblem} Let $G$ be a simple algebraic group of exceptional type. Describe the poset of conjugacy classes of semisimple subgroups of $G$.\end{mainproblem}

Just as in representation theory, where one may begin by studying irreducible modules for a given object (immediately yielding the completely reducible modules) and then considering extensions, one can stratify the search for subgroups of $G$, beginning with $G$-cr subgroups and building up from these to non-$G$-cr subgroups.

Suppose that we know the maximal connected subgroups of simple groups up to a certain rank (such as $8$), and let $G$ be one of these simple groups. Let $H\subseteq G$ be semisimple and let $M$ be a maximal connected subgroup of $G$ containing $H$. If $M$ is reductive then $M=\D(M)\cdot Z(M)$ and since $H$ is perfect, we have $H\subseteq \D(M)$. Since $\D(M)$ is a central product of simple groups of smaller dimension, we hope to know $H$ by induction on the dimension. The problem is that one may have $H\subset M$ where $M$ is not reductive; this means $M=P$ is a parabolic subgroup by the Borel--Tits theorem. So we would also like to know the semisimple subgroups of $P=QL$. Since $L$ is reductive of the same rank as $G$ we may again assume that we know its semisimple subgroups. But it remains to find those semisimple subgroups of $P$ which are not conjugate to subgroups of $L$; in other words, the non-$G$-cr subgroups. Of course, Theorem \ref{thm:gcr-red} tells us that such subgroups do not exist if $\Char k$ is $0$ or large enough relative to the root system of $G$.

\subsection{\texorpdfstring{$G$}{G}-cr subgroups} \label{sec:gcralg}
To start, we need to know the $G$-conjugacy classes of $G$-cr semisimple subgroups of $G$. Theorem~\ref{thm:gcr-strong-red} reduces our task to finding those which are $L$-irr, as $L$ varies over all Levi subgroups of $G$. Let $\{L_1,\dots, L_s\}$ be a complete list of representatives of the $G$-conjugacy classes of Levi subgroups. Then we can find all conjugacy classes of semisimple $G$-cr subgroups at least once, by listing the $L_i$-conjugacy classes of $L_i$-irr subgroups. The full analogue of the Jordan--H\"older theorem given below says that this will give each $G$-class of $G$-cr subgroups exactly once. It can be proven in various ways (cf.~\cite[Propositions~2.8.2, 2.8.3]{Car93}) but the cleanest is via geometric invariant theory \cite[Theorem~5.8]{MR3042598}.

\begin{lemma} \label{lem:h-to-pi-h}
Let $H$ be a subgroup of the connected reductive group $G$, let $P$ be minimal amongst parabolic subgroups of $G$ containing $H$ and let $\pi : P \to P/\RR_u(P) = L$ be the natural projection to a Levi subgroup $L$. Then $\pi(H)$ is $L$-irr, and the $G$-conjugacy classes of $L$ and $\pi(H)$ are uniquely determined by the $G$-conjugacy class of $H$.

Moreover, $H$ is $G$-cr if and only if $H$ and $\pi(H)$ are $\RR_u(P)$-conjugate.
\end{lemma}

As $L$ contains a maximal torus of $G$, the group $N_G(L)$ is an extension of $L$ by a finite group: the part of the Weyl group of $G$ which stabilises the root system of $L$. Thus $N_G(L)$ may induce non-trivial conjugacy between some simple factors of $L$. Such conjugacy is easy to describe. cf.~\cite[Corollary~12.11]{MR2850737}. In light of this, the key issue is to find the $\D(L)$-classes of $\D(L)$-irr semisimple subgroups for each $G$-class of Levi subgroup $L$; this includes the case that $L=G$. 

\begin{defn} \label{def:irr-g}
For a connected reductive group $G$, let $\mathrm{ConIrr}(G)$ denote the poset of $G$-classes of connected $G$-irr subgroups of $G$ under inclusion.
\end{defn}

Suppose that $[H]\in \mathrm{ConIrr}(G)$. It follows from the Borel--Tits theorem that $H$ is reductive. Since it cannot centralise a non-central torus of $G$, we get:
\begin{lemma}[{\cite[Lemma~2.1]{MR2043006}, \cite[Corollary~3.18]{MR2178661}}]
Suppose $G$ is semisimple and let $H$ be a $G$-irr connected subgroup. Then $H$ is semisimple and $C_G(H)$ is finite and linearly reductive. 
\end{lemma}

\begin{remark}Work of the third author and Liebeck in \cite{LieTho} classifies those finite subgroups which can occur as centralisers of semisimple $G$-irr subgroups for any simple algebraic group $G$.\end{remark}

As $H$ is not contained in any proper parabolic subgroup of $G$, it must be contained in some semisimple maximal connected subgroup $M$ of $G$. Moreover $H$ is $M$-irr, since any proper parabolic subgroup of $M$ is contained in a proper parabolic subgroup of $G$, again by the Borel--Tits theorem. When $M$ is simple of classical type, one may determine $\mathrm{ConIrr}(M)$ at once by use of Proposition \ref{prop:irredclass}. If instead $M$ is simple of exceptional type then induction on $\dim G$ yields $\mathrm{ConIrr}(M)$. For general semisimple $M$, knowledge of $\mathrm{ConIrr}(M_i)$ for each simple factor $M_i$ of $M$ yields $\mathrm{ConIrr}(M)$ (more on this shortly). Now, it can happen that an $M$-irr subgroup $H$ lies in a proper parabolic subgroup of $G$. We call $[H] \in \mathrm{ConIrr}(M)$ a \emph{candidate} and aim to decide which candidates are actually $G$-irr. Also, a $G$-irr subgroup can be contained in more than one semisimple maximal connected subgroup $M$. To detect this, we want to know when two candidates $H_1\subseteq M_1$ and $H_2\subseteq M_2$ are in the same $G$-class.

Returning to the issue of finding $\mathrm{ConIrr}(M)$ when $M$ is not simple, the parabolic subgroups of $M$ are the products of parabolic subgroups of its factors, so any $M$-irr subgroup needs to project to an $M_i$-irr subgroup of each simple factor $M_i$ \cite[Lemma~3.6]{Tho1}. As a partial converse, if for each $i$ we have an $M_i$-irr subgroup $H_i$, then $H_1 \ldots H_r\subseteq M$ is $M$-irr; however, these do not quite exhaust all the $M$-irr subgroups. To complete the list, one must also discuss \emph{diagonal subgroups}: Whenever $M$ has one or more simple factors of a given type, let $\hat H$ be the simply-connected simple group of this type. Then $\hat H$ admits a homomorphism to $M$ with non-trivial projection to each simple factor of the of the appropriate type. By definition, a diagonal subgroup is the commuting product of images of such homomorphisms; this will be $M$-irr precisely when it has non-trivial projection in every simple factor of $M$.

\begin{example} \label{ex:diagonalA1A1}
Let $M = M_1M_2$ where the factors are simple of the same type. Then $M$ has a diagonal subgroup $H$ with simply-connected cover $\hat H$. Then the composed maps $\hat H \to H \to M_i$ are isogenies. It follows that up to $M$-conjugacy, these compositions of powers of Frobenius maps $F$ (or their square roots in some very special cases) with automorphisms which induce a symmetry of the Dynkin diagram of $\hat H$. For instance, if $M$ has type $A_1A_1$ then, since the $A_1$ Dynkin diagram has trivial symmetry group, $H$ corresponds to a pair of non-negative integers $(r,s)$ and the map $\hat H\to H$ is $x\mapsto (F^r(x),F^s(x))$. Since $H\cong F(H)$ we may assume $rs=0$. For brevity, we use the notation $A_1 \hookrightarrow A_1 A_1$ via $(1^{[r]},1^{[s]})$ for these diagonal subgroups. See \cite[Chapters~2,11]{ThomasIrreducible} for further discussion and notation.
\end{example}

To recap, our recipe is now to iterate through the maximal connected subgroups $M$ of $G$, collecting (semisimple) candidates $H$. We throw away all those candidates which fall into a proper parabolic subgroup of $G$, and then determine the poset $\mathrm{ConIrr}(G)$ by identifying conjugacy amongst the remaining candidates.

The passage from $G$-irr semisimple subgroups to all $G$-cr semisimple subgroups is now easy using Lemma~\ref{lem:h-to-pi-h}. Suppose that $H$ is $\D(L)$-irr. By the above remarks we can write down the $\D(L)$-classes of $\D(L)$-irr subgroups; then from the lemma we need only establish conjugacy amongst those classes by examining the action of the stabiliser of $L$ in the Weyl group of $G$.

\begin{remark}
A related representation-theoretic question is to classify triples $(G,H,V)$ where $H \subseteq G$ and $V$ is an irreducible $G$-module which remains irreducible as an $H$-module. This property is strictly stronger than $G$-irreducibility and has its own extensive literature, cf.~\cite{MR3931415} and the references therein.
\end{remark}

\subsection{Non-\texorpdfstring{$G$}{G}-cr subgroups} \label{s:nonGcrstrat}
We now turn our attention to non-$G$-cr semisimple subgroups $H$ (recalling again that $G$ is connected, reductive). Then $H\subseteq P$ for some proper parabolic subgroup $P=QL$ with $Q=\RR_u(P)$ and we may assume $P$ is minimal subject to containing $H$. Let $\bar H$ denote the image of $H$ in $L$ under the projection $\pi:P\to L$. Then by Lemma~\ref{lem:h-to-pi-h} $\bar H$ is $L$-irr (hence $G$-cr) and is not $\RR_u(P)$-conjugate to $H$. We may assume from the previous section that we know $\bar H$ up to conjugacy. To make further progress, we use non-abelian cohomology, whose techniques are similar to those employed in Galois cohomology.

\label{bn_cn}Firstly, note that either $\pi: H\to \bar H$ is an isomorphism of algebraic groups or a very special situation occurs, namely, $p = 2$ and $H$ has a simple factor $\SO_{2n+1}$ with image $\Sp_{2n}$ in $\bar H$, so that the differential $d\pi: H \to \bar H$ has a non-zero kernel.\footnote{On the level of schemes, $H$ intersects $Q$ non-trivially, giving rise to a non-zero scheme-theoretic kernel. See \cite{BT73} for the theory surrounding this map, or one of \cite[Lemma~2.2]{PY06}, \cite{Vas05}, \cite{DS96} for more concrete treatments.}

Suppose for now that this special situation does not hold, so that $H$ and $\bar H$ are isomorphic as algebraic groups and $H$ is a complement to $Q$ in the semidirect product $ \bar H$. Any element of $H$ can thus be written uniquely as $\gamma(h)h$ with $h\in\bar H$, for some map $\gamma:\bar H\to Q$ which is a morphism of varieties. The definition of the semidirect product implies that $\gamma$ satisfies a $1$-cocycle condition; namely:
\[ \gamma(gh)= \gamma(g) (g \cdot \gamma(h)).\]
If $H'$ is another complement to $Q$, corresponding to a map $\gamma'$, then $H$ is $Q$-conjugate to $H'$ if and only if $\gamma$ is related to $\gamma'$ via a coboundary; in other words the $Q$-conjugacy classes of complements are given by classes $[\gamma]\in\opH^1(\bar H,Q)$. We leave the description of the precise relationship of $\gamma$ and $\gamma'$ to \cite[\S 2]{SteUni}, but note that since $Q$ is typically non-abelian, the set $\opH^1(\bar H,Q)$ does not admit the structure of a group---rather, it is only a pointed set, having a distinguished element corresponding to the class of the trivial cocycle. In contrast, when $Q$ has the structure of an $\bar H$-module---i.e.~$Q$ is a vector space on which $\bar H$ acts linearly---then both $Q$ and $\opH^1(\bar H,Q)$ are naturally $k$-vector spaces. In the latter case it can be shown that $\opH^1(\bar H,Q)$ is isomorphic to the first right-derived functor (applied to $Q$) of the fixed point functor $\opH^0(\bar H,?)$ in the category of rational $G$-modules. For more on this last point, see \cite[I.4]{Jan03}.

\begin{example}
By way of illustration, we list the semisimple subgroups of $G = \SL_3$. There are four parabolic subgroups of $G$ up to conjugacy, respectively stabilising flags with submodule dimensions $(3)$, $(2,1)$, $(1,2)$ and $(1,1,1)$. The first is $G$ itself, the last is a Borel subgroup (whose only reductive subgroups are tori) and the other two have $G$-conjugate Levi subgroups $\GL_2$, one of which can be described as the image of the embedding \[\GL_2 \to \SL_3, \quad A \mapsto \left(\begin{array}{c|c} A & 0 \\ \hline 0 & \det(A)^{-1} \end{array}\right).\]

Of course $G$ is $G$-irr; and if $L$ is a Levi subgroup isomorphic to $\GL_2$ then $\D(L)\cong\SL_2$, which has no proper semisimple subgroups. If $p\neq 2$, then there is one further $G$-irr subgroup $\PGL_2$, embedded via the irreducible adjoint action on its Lie algebra. These are all the $G$-cr semisimple subgroups.

Suppose that $H\subseteq G$ is semisimple and non-$G$-cr. By rank considerations, $\bar H$ is isomorphic to $\SL_2$, lying in a parabolic subgroup $P = QL$ with $L \cong \GL_2$. It is easy to check that the conjugation action of $\bar H$ on $Q$ gives it the structure of the natural module $L(1)$. This means $\opH^1(\bar H,Q)=0$, which rules out the existence of non-$G$-cr subgroups unless $p=2$.

If $p = 2$, however, the action of $\SL_2$ on its Lie algebra and its dual are not completely reducible: Both are indecomposable with two composition factors. The first is isomorphic to the Weyl module $V(2)$ which has $L(2)$ in the head and $L(0)$ in its socle; we denote this $L(2)/L(0)$. The second is upside down: $\opH^0(2)\cong L(0)/L(2)$. This yields two non-conjugate, non-$G$-cr subgroups $\PGL_2$, one in each of the two standard parabolic subgroups of $G$.\footnote{In fact, this is the `special situation' mentioned earlier, since $H = \SL_2$ is abstractly isomorphic to its image $\PGL_2$ in $G = \SL_3 = \SL(\Lie(\SL_2))$.}

Concluding, the subgroups above---$G$ itself, the two derived Levi subgroups $\SL_2$, a $G$-irr subgroup $\PGL_2$ (when $p \neq 2$) and two non-$G$-cr subgroups $\PGL_2$ (when $p = 2$)---are now all the non-trivial semisimple subgroups of $G$.
\end{example}

\subsection{Abelian and non-abelian cohomology} \label{ss:cohom}

Let $G$ be a connected reductive group acting on a $G$-module $V$. To mount a proper investigation of $\opH^1(G,V)$, a scheme-theoretic treatment such as \cite{Jan03} is essential. This is not least because one can make use of the Lyndon--Hochschild--Serre spectral sequence
\[ E_2^{ij}=\opH^i(G/N,\opH^j(N,V))\Rightarrow \opH^{i+j}(G,V)\] for calculations, where $N$ is a normal subgroup scheme of $G$. In this framework, $N$ is allowed to be an \emph{infinitesimal subgroup scheme}, the most important example being the Frobenius kernel $G_1:=\ker(G\to F(G))$ of $G$. At the level of points, $G_1=\{1\}$, but $\Lie(G_1)=\Lie(G)$ has far more structure. We leave the interested reader to pursue this further, but give some references: The first general investigation of $\opH^1(G,V)$ using the LHS spectral sequence applied to $G_1\triangleleft G$ is probably that of Jantzen in \cite{Jan91}, which connects $\opH^1(G,L(\lambda))$ with the structure of the Weyl module $V(\lambda)$. Other relevant papers are too numerous to mention, but some highlights are \cite{CPSV77}, \cite{BNP04-Frob}, \cite{BNPPSS}, \cite{Par07}.

On the understanding that $\opH^1(\bar H,V)$ has been well-studied for $V$ an $\bar H$-module, let us return to the calculation of $\opH^1(\bar H,Q)$, where $P = QL$ is a parabolic subgroup of a reductive algebraic group $G$. The fact that $Q$ is connected, smooth and unipotent means that it admits a filtration
\[ Q = Q_0 \supseteq Q_1 \supseteq \cdots \supseteq Q_{n} = 1 \]
for some $n$, such that $Q_i\triangleleft Q$ and the subquotients $Q_i/Q_{i+1}$ admit the structure of $\bar H$-modules. The statement for general connected unipotent groups $Q$ can be found in \cite[Theorem~3.3.5]{SteUni} and \cite[Theorem~C]{McN14}, but one can be more explicit here since $Q$ has a filtration by subgroups $Q_i$ generated by root subgroups of $G$. Following \cite{ABS90}, let $P$ and $L$ be a standard parabolic and Levi subgroup, corresponding to a subset $I$ of the simple roots $\Delta$ (which can be identified with nodes of the Dynkin diagram of $G$). Then $P$ is generated by a maximal torus and root subgroups $U_{\alpha}$ where $\alpha$ runs through positive roots, as well as negative roots in $I$. Expressing each root $\alpha$ uniquely as
\[\alpha = \left(\sum_{\alpha_i \in I} c_i \alpha_i\right) + \left(\sum_{\alpha_j \in \Delta \setminus I} d_j \alpha_j\right), \]
the roots in $P$ are those with $\sum d_i \ge 0$; the roots occurring in $L$ are those with $\sum d_i = 0$, and those in $Q$ have $\sum d_i > 0$. The quantity $\sum d_j$ is called the \emph{level} of the root, and $\sum d_j \alpha_j$ is called its \emph{shape}. For each $i > 0$, denote by $Q_i$ the subgroup generated by root subgroups of level $i$; then the Chevalley commutator relations imply that each $Q_i$ is normal in $Q$, and in fact $Q_i/Q_{i+1}$ is central in $Q/Q_{i+1}$. Furthermore, from knowledge of the root system of $G$, say by reference to \cite{Bourb05}, one can write down explicitly the representations $Q_i/Q_{i+1}$ as Weyl modules $V(\lambda)$ for the Levi subgroup $L$. 

\begin{example} Recall that the Dynkin diagram of $G_2$ is \dynkin[%
      label,
      reverse arrows,
      text style/.style={scale=0.5},
      label macro/.code={\drlap{#1}}]G2. The nodes represent the two simple roots which we denote by lists of coefficients, with $\alpha_1=10$ and $\alpha_2=01$. The remaining positive roots, in order of height, are $11$, $21$, $31$, $32$. Let $P$ be the standard parabolic containing the negative of $\alpha_2$, i.e.~$-01$. Then a Levi factor $L$ of $P$ has roots $\pm 01$, and $\RR_u(P)$ has roots of three levels $\{11,10\}$, $\{21\}$ and $\{31,32\}$, which one checks induce modules $L(1)$, $L(0)$, $L(1)$ for the Levi subgroup $\D(L)\cong\SL_2$.\end{example}

Once the modules $Q_i/Q_{i+1}$ and cohomology groups $\opH^{1}(\bar H,Q_i/Q_{i+1})$ are understood, one can take the direct sum $\mathbb V:=\bigoplus\opH^1(\bar H,Q_i/Q_{i+1})$ and use this to approximate $\opH^1(\bar H,Q)$. In fact, one can define a partial map $\mathbb V\to \opH^1(\bar H,Q)$, which turns out to be surjective, using a lifting process we now describe. Given any short exact sequence of $\bar H$-groups
\[1\to R\to Q\to S\to 1\]
with $R$ contained the centre of $Q$, there is an exact sequence of $\bar H$-sets\footnote{Here, an exact sequence of pointed sets means only that the image of each map is the preimage of the distinguished element under the next.}
\begin{equation}
\begin{split}
  1\to &\opH^0(\bar H,R)\to \opH^0(\bar H,Q)\to \opH^0(\bar H,S)\label{eq:long}\\
  &\stackrel{\delta_{\bar H}}{\to} \opH^1(\bar H,R)\to \opH^1(\bar H,Q)\to \opH^1(\bar H,S)\stackrel{\Delta_{\bar H}}{\to} \opH^2(\bar H,R).
\end{split}\end{equation}
Now, taking $R = Q_{i}/Q_{i+1}$ and $S = Q/Q_{i}$ for each $i$, one can use these `long' exact sequences to lift elements of $\opH^1(\bar H,Q/Q_{i})$ to elements of $\opH^1(\bar H,Q/Q_{i+1})$, eventually reaching $\opH^{1}(\bar H, Q)$ itself, as long as we understand two issues:
\begin{enumerate}[label=(\roman*)]
  \item When is $\opH^1(\bar H,Q/Q_i)\to\opH^1(\bar H,Q)$ not injective? \label{q1}
  \item When is $\opH^1(\bar H,Q/Q_i)\to\opH^1(\bar H,Q/Q_{i+1})$ not defined? \label{q2}
\end{enumerate}
Question \ref{q1} asks whether $\delta_{\bar H}$ is non-zero. This happens precisely when cocycle classes in $\opH^1(\bar H,R)$ fuse inside $\opH^1(\bar H, Q)$ due to conjugacy induced by the fixed points $S^{\bar H}$. Question~\ref{q2} asks whether $\Delta_{\bar H}$ is non-zero. If so then cocycles in $\opH^1({\bar H},S)$ are obstructed from lifting to cocycles in $\opH^1({\bar H},Q)$. In particular, this only happens when $\opH^2({\bar H},R)\neq 0$.

In the end, this lifting process allows us to calculate $\opH^1(\bar H,Q)$ completely. The matter is easy if we can show the maps $\delta_{\bar H}$ and $\Delta_{\bar H}$ to be zero. However, this is often not the case and one must resort to explicit computations with cocycles; this is the approach taken in \cite{SteF4}.

\addtocounter{chapter}{1}
\chapter*{Part II. Subgroup structure of exceptional algebraic groups}

\label{sec:subgroups}

\setcounter{section}{0}
\section{Maximal subgroups}

Work on classifying maximal sub-objects of Lie type objects dates back to Sophus Lie \cite{Lie80}. Taking inspiration from Galois's work on univariate polynomials, \emph{op.~cit.} develops `continuous transformation groups'---now Lie groups---with a view to classifying differential equations in terms of symmetries amongst their solutions. One builds up group actions from primitive actions, corresponding to maximal subgroups, motivating Lie to describe such subgroups. The same problem for finite groups was not to be posed until a paper of Aschbacher and Scott \cite{MR772471} rather later, and Lie concentrated on connected subgroups of connected Lie groups. Here, the $\exp$ and $\log$ make this equivalent to finding maximal subalgebras $\m$ of real Lie algebras $\g$, and Lie solved the problem when $\dim\g\leq 3$. Otherwise the question lay dormant for another fifty years.

Using the Killing--Cartan--Weyl classification of finite-dimensional complex simple Lie algebras, E.~Dynkin solved Lie's problem over $\mathbb{C}$ \cite{ebd}. We give a quick example---stolen from Seitz's excellent tribute in \emph{op.~cit.}---to illustrate his results. As is well-known, the complex $3$-dimensional Lie algebra $\sl_2$ has a unique irreducible representation of each degree up to equivalence. This amounts to an embedding of $\sl_2$ into $\so_{2n-1}$ or $\sp_{2n}$, where $2n-1$ or $2n$ respectively is the degree. Dynkin showed that for $n\geq 2$, the image of each of these embeddings is a maximal subalgebra, with precisely one exception: when $n=7$ and the exceptional Lie algebra of type $G_2$ has a self-dual $7$-dimensional module, it occurs as a (maximal) subalgebra of $\so_7$ and in turn contains the irreducible $\sl_2$ as a maximal subalgebra. There is a remarkably short list of such situations. Dynkin in effect classified the maximal subalgebras of the classical Lie algebras $\sl_{n+1}$, $\so_n$ and $\sp_{2n}$ by classifying non-maximal ones which nevertheless act irreducibly on the natural modules for those algebras. A key ingredient in Dynkin's work was detailed information on the representations of these Lie algebras, developed by Weyl and others, in terms of the weights for their Cartan subalgebras.

Dealing with the Lie algebras of exceptional type required Dynkin to adopt a more exhaustive approach. He first showed how to produce all the semisimple subalgebras of $\g=\Lie(G)$ containing a given Cartan subalgebra $\h$, so-called \emph{regular} subalgebras. Since root spaces are $1$-dimensional, it follows that such a subalgebra will be the sum of $\h$ and the root spaces corresponding to a subset $\Phi'$ of the root system $\Phi$ of $G$. %
Dynkin showed that one can find all regular subalgebras by iteratively extending the Dynkin diagram (adding a node corresponding to the negative of the highest long root) and then deleting some nodes.

\begin{example}
Let $\Phi$ be an irreducible root system of type $F_4$ with roots labelled as in the following diagram.  
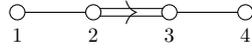
\begin{center}
\dynkin[label]F4
\end{center}
Then the highest long root is $\alpha_0 =  2 \alpha_1 + 3\alpha_2 + 4\alpha_3 + 2\alpha_4$. The only simple root one can add to $-\alpha_0$ and still get a root is $\alpha_1$. Therefore, the extended Dynkin diagram is:
\begin{center}
\dynkin[label, extended]F4
\end{center}
The maximal subalgebras of maximal rank correspond to deleting a node of this extended diagram corresponding to a simple root with prime coefficient in the expression of $\alpha_0$: in our case, this is $\alpha_1, \alpha_2$ and $\alpha_4$. Removing $\alpha_1$ gives a Dynkin diagram of type $A_1 C_3$, removing $\alpha_2$ gives type $A_2A_2$ and removing $\alpha_4$ gives type $B_4$. 
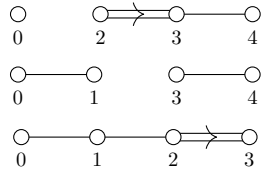
\begin{center}
\dynkin[labels={0}]A1 \ \ \ \ \ \begin{dynkinDiagram}[labels={2,3,4}]A3 \dynkinDefiniteDoubleEdge 12 \end{dynkinDiagram} 

\dynkin[labels={0,1}]A2 \ \ \ \ \ \dynkin[labels={3,4}]A2

\dynkin[labels={0,1,2,3}]B4 
\end{center}
\end{example}

Non-semisimple subalgebras were described by a theorem of Morozov\footnote{the precursor of Borel--Tits' Theorem~\ref{thm:bt}.} and the maximal ones are the maximal parabolics. This leaves those maximal subalgebras which do not contain a Cartan subalgebra, so-called \emph{$S$-subalgebras}. Dynkin tackled those of type $\sl_2=\langle e,f,h\rangle$ first, associating to each class of these under the adjoint action of $G$ a Dynkin diagram with a label of $0$, $1$ or $2$ above each node determining the conjugacy class of $h$. It turns out there is a unique conjugacy class of $\sl_2$-subalgebras such that $h$ is a \emph{regular element}, i.e.~the centraliser $\g_h$ of $h$ in $\g$ is as small as possible, that is, $\g_h$ is a Cartan subalgebra. The corresponding Dynkin diagram for this class of subalgebras has a $2$ above each node and it is usually maximal. From there, if $\sl_3$ is a subalgebra of $\g$, then one can look to build it up from its own regular $\sl_2$.

There are many reasons to extend this theory to positive characteristic, not least because algebraic groups and their points over finite fields give information about finite groups, for instance in furtherance of the Aschbacher--Scott programme. Over several important monographs of Seitz \cite{MR888704,MR1048074}, Liebeck--Seitz \cite{MR2044850} and Testerman \cite{MR961210}, Dynkin's classification is extended to describe the maximal subgroups of simple algebraic groups over algebraically closed fields of positive characteristic.

Unfortunately, there is no space to do anything else but state the main result in case when $G$ is exceptional. In the following, conditions such as $p \ge 13$ also include the case $p = 0$. Note that $\tilde{H}$ denotes a subgroup of type $H$ whose root groups are generated by short roots of $G$. 

\begin{theorem}[{\cite[Corollary~2]{MR2044850}}, {\cite[Theorem.~1]{MR4375734}}] \label{t:maximalexcep}
Let $G$ be a simple algebraic group of exceptional type in characteristic $p$ and let $M$ be maximal amongst connected subgroups of $G$. Then $M$ is either parabolic or is $G$-conjugate to precisely one subgroup $H$ in Table~\ref{tab:maximalexcep}, where each $H$ denotes one $G$-conjugacy class of subgroups.
\end{theorem}

\begin{longtable}{ll}
\caption{The reductive maximal connected subgroups of exceptional algebraic groups.} \label{tab:maximalexcep} \\
\toprule $G$ & $H$ \\ \midrule
$G_2$ & ${A}_2$, $\tilde{A}_2$ $(p =3)$, $A_1 \tilde{A}_1$, $A_1$ $(p \geq 7)$ \\
$F_4$ & $B_4$, $C_4$ $(p=2)$, ${A}_1 C_3$ $(p \neq 2)$, $A_1 G_2$ $(p \neq 2)$, $A_2 \tilde{A}_2$, \\ & $G_2$ $(p = 7)$, $A_1$ $(p \geq 13)$ \\
$E_6$ & ${A}_1 A_5$, ${A}_2^3$, $F_4$, $C_4$ $(p \neq 2)$, $A_2 G_2$, $G_2$ $(p \neq 7; 2$ classes), \\&  $A_2$ $(p \geq 5; 2$ classes) \\
$E_7$ & ${A}_1 D_6$, ${A}_2 A_5$, $A_7$, $G_2 C_3$, $A_1 F_4$, $A_1 G_2$ $(p \neq 2)$, $A_2$ $(p \geq 5)$, \\& $A_1 A_1$ $(p \geq 5)$, $A_1$ $(p \geq 17)$, $A_1$ $(p \geq 19)$ \\
$E_8$ & $D_8$, ${A}_1 E_7$, ${A}_2 E_6$, $A_8$, ${A}_4^2$, $G_2 F_4$, $F_4$ $(p = 3)$, $B_2$ $(p \geq 5)$, \\ & $A_1 A_2$ $(p \geq 5)$, $A_1$ $(p \geq 31)$, $A_1$ $(p \geq 29)$, $A_1$ $(p \geq 23)$ \\ \bottomrule
\end{longtable}

\begin{remarks} \leavevmode
\begin{enumerate}
  \item As discussed in \S\ref{parabsec}, the classes of parabolic subgroups are in bijection with subsets of the simple roots $\Delta$, and the maximal ones correspond to subsets of size $|\Delta| -1$.  

\item In the caption of Table \ref{tab:maximalexcep} we use the phrase {\it reductive maximal connected}. By this mean we reductive subgroups which are maximal amongst connected subgroups. Similarly, in the caption to the next table we say {\it reductive maximal positive-dimensional} to mean reductive subgroups which are maximal amongst positive-dimensional subgroups. 

\item The subgroups of maximal rank can be enumerated using the Borel--de-Siebenthal algorithm. This is a more general version of Dynkin's procedure for regular subalgebras, and includes some extra cases where the Dynkin diagram has an edge of multiplicity $p = \Char k$. For example if $G$ is of type $F_4$, then there is a maximal regular subgroup of type $C_4$. See \cite[\S 13.2]{MR2850737} for a complete explanation. 

\item When $G = E_6$ there are two classes of maximal subgroups of type $G_2$ ($p \geq 7$) and $A_2$ ($p \geq 5$). The graph automorphism of $G$ interchanges these two classes. See \cite{MR1100666} for a proof of this and an explicit construction of the maximal subgroups.   

\item The maximal subgroup of type $F_4$ when $G=E_8$ and $p=3$ was overlooked in \cite{MR1048074} and subsequently missed in \cite{MR2044850}. This was rectified by Craven and the second two authors; more information can be found in \cite{MR4375734}.
\end{enumerate}
\end{remarks}

It is also natural to ask about non-connected maximal subgroups. For example, finite subgroups of $E_8$ remain unclassified. See \S\ref{ss:maxfinsub} for a brief description of the latest developments. However one can successfully weaken `connected' to `positive-dimensional': 

\begin{theorem}[{\cite[Corollary~2]{MR2044850}}] \label{t:posdimmaximalexcep}
Let $G$ be a simple algebraic group of exceptional type in characteristic $p$. Let $M$ be a positive-dimensional maximal subgroup of $G$. Then $M$ is either parabolic or $G$-conjugate to precisely one subgroup $H$ as follows. Each isomorphism type of $H$ denotes one $G$-conjugacy class of subgroups, and the notation $T_i$ indicates an $i$-dimensional torus.

\clearpage

\begin{longtable}{ll}
\caption{The reductive maximal positive-dimensional subgroups of exceptional algebraic groups.} \label{tab:maxexcepposdim} \\
\midrule $G$ & $H$ \\ \midrule
$G_2$ & ${A}_2.2$, $\tilde{A}_2.2$ $(p =3)$, ${A}_1 \tilde{A}_1$, $A_1$ $(p \geq 7)$ \\

$F_4$ & $B_4$, $D_4.S_3$, $C_4$ $(p=2)$, $\tilde{D}_4.S_3$ $(p=2)$, ${A}_1 C_3$ $(p \neq 2)$, \\ & $A_1 G_2$ $(p \neq 2)$, $({A}_2 \tilde{A}_2).2$, $G_2$ $(p = 7)$, $A_1$ $(p \geq 13)$ \\

$E_6$ & ${A}_1 A_5$, $({A}_2^3).S_3$, $(D_4T_2).S_3$, $T_6.W(E_6)$, $F_4$, $C_4$ $(p \neq 2)$, $A_2 G_2$, \\ & $G_2$ $(p \neq 7$), $A_2.2$ $(p \geq 5)$ \\

$E_7$ & ${A}_1 D_6$, $({A}_2 A_5).2$, $A_7.2$, $(A_1^3D_4).S_3$, $(A_1^7).\text{PSL}_3(2)$, $(E_6T_1).2$, \\ & $T_7.W(E_7)$, $G_2 C_3$, $A_1 F_4$, $(2^2 \times D_4).S_3$,  $A_1 G_2$ $(p \neq 2)$, \\ & $A_2.2$ $(p \geq 5)$, $A_1 A_1$ $(p \geq 5)$, $A_1$ $(p \geq 17)$, $A_1$ $(p \geq 19)$ \\

$E_8$ & $D_8$, ${A}_1 E_7$, $({A}_2 E_6).2$, $A_8.2$, $({A}_4^2).4$, $(D_4^2).(S_3 \times 2)$, \\ & $(A_2^4).(\text{GL}_2(3))$, $(A_1^8).\text{AGL}_3(2)$, $T_8.W(E_8)$, $G_2 F_4$, \\ & $A_1 (G_2^2).2$ $(p \neq 2)$,   $F_4$ $(p = 3)$, $B_2$ $(p \geq 5)$, $A_1 A_2$ $(p \geq 5)$, \\ & $A_1$ $(p \geq 31)$, $A_1$ $(p \geq 29)$, $A_1$ $(p \geq 23)$, $A_1 \times S_5$ $(p \geq 7)$ \\ \midrule
\end{longtable}
\end{theorem}

\begin{remark}
The subgroup $A_2 G_2 < E_6$ is a maximal subgroup and its presence above corrects a small mistake in \cite[Table~1]{MR2044850}. In \textit{loc.~cit.} it is claimed that $N_{E_6}(A_2G_2) = (A_2G_2).2$ with the outer involution acting as a graph automorphism of the $A_2$ factor. This is not possible as the action of $A_2 G_2$ on the $27$-dimensional $E_6$-module $V_{27}$ is not self-dual. Instead, it is the graph automorphism of $E_6$ which induces an outer involution on $A_2 G_2$.
\end{remark}

\section{The connected \texorpdfstring{$G$}{G}-irreducible subgroups} \label{sec:gcrsubs}
In light of \S\ref{sec:gcralg} there are three things we need to determine $\mathrm{ConIrr}(G)$ (Definition~\ref{def:irr-g}) for $G$ a simple exceptional algebraic group. 

\begin{enumerate}[label=(\roman*)]
\item Determine the semisimple maximal connected subgroups of $G$; \label{step:ss-max}

\item Decide whether a candidate subgroup\footnote{Recall that a candidate subgroup is an $M$-irr subgroup from a semisimple maximal connected subgroup $M$.} is $G$-irr; \label{step:cand}

\item Decide whether two $G$-irr candidate subgroups are $G$-conjugate. \label{step:conj}
\end{enumerate}

Since $G$ is a simple exceptional algebraic group, \ref{step:ss-max} is immediate from Theorem \ref{t:maximalexcep}. We consider \ref{step:cand} and \ref{step:conj} in the next two sections. 

\subsection{Testing candidate subgroups}
Let $H$ be an $M$-irr connected subgroup of $G$, where $M$ is maximal semisimple. We need to decide whether or not $H$ is $G$-irr.

\subsubsection{Proving that candidates are $G$-irr}

If a candidate $H$ is in fact contained in a parabolic subgroup $P=QL$ of $G$ then we can consider the image $\pi(H)$ in a Levi subgroup. The action of $H$ and $\pi(H)$ on $G$-modules may differ but their composition factors will always match.\footnote{see \cite[Lemma~3.8]{Tho1} for the precise definition of {\it match}.} Thus one way to prove that $H$ is $G$-irr is to show that its composition factors on some $G$-module do not match those of any proper Levi subgroup of $G$. For instance, every proper Levi subgroup of $G$ has a trivial composition factor on the adjoint module $\Lie(G)$, so a candidate is $G$-irr if it has no trivial composition factors on $\Lie(G)$.

\subsubsection{Proving that candidates are not $G$-irr}

Now let $H$ be a candidate subgroup which we believe is not $G$-irr. If $H$ is $G$-cr but not $G$-irr then $H$ is contained in some proper Levi subgroup $L$ of $G$ by Lemma~\ref{lem:h-to-pi-h}, and thus $C_G(H)$ will contain the non-trivial torus $Z(L)^\circ$. In this case, it is often easy to find a non-trivial torus commuting with $H$ and thus conclude that $H$ is not $G$-irr. 

The most difficult cases are when a candidate $H$ turns out to be non-$G$-cr (and thus not $G$-irr). Such cases are relatively rare: \cite[Corollary~3]{ThomasIrreducible} classifies the non-$G$-cr connected subgroups which are $M$-irr for every (and at least one) reductive maximal connected subgroup $M$ in which they are contained. There are two main methods used in \cite{Tho1,Tho2}. Briefly, one either 
\begin{enumerate}[label=(\arabic*)]
\item directly shows that $H$ is contained in a parabolic subgroup $P$; or \label{method1}
\item finds a non-$G$-cr subgroup $Z \subset P$ and show that $Z$ is contained in $M$ and conjugate to $H$. \label{method2}
\end{enumerate}

One way to implement \ref{method1} is to exhibit a non-zero fixed point of $H$ on the adjoint module $\text{Lie}(G)$. By \cite[Lemma~1.3]{MR1048074}, this places $H$ in either a proper maximal-rank subgroup or a proper parabolic subgroup of $G$, and one can use representation theory to prove that $H$ is not contained in a proper maximal-rank subgroup. Another way is to find a unipotent subgroup of $G$ normalised by $H$, since the Borel--Tits theorem then places $H$ in a proper parabolic subgroup. This is used in \cite[Lemmas~7.9, 7.13]{Tho1}, where calculations in Magma are used to construct an ad-nilpotent subalgebra $S \subset \text{Lie}(G)$ stabilised by a `large enough' finite subgroup $H(q) < H$, where $q = p^r$ for some $r > 0$. It then follows that $H$ also stabilises $S$, and one checks that one can exponentiate $S$ to yield a unipotent subgroup normalised by $H$.   

Method \ref{method2} is implemented in \cite[Lemmas~6.3, 7.4]{Tho1}. Here one starts with a candidate $H$ contained in a semisimple maximal connected subgroup $M$ of maximal rank. One constructs the relevant non-$G$-cr subgroup $Z$ according to the recipe in \S\ref{s:nonGcrstrat}, and then shows that $Z\subset M$. To establish the latter, one proves that any group acting with the same composition factors as $Z$ on the adjoint module of $G$ fixes a non-zero element of $\text{Lie}(G)$. One then shows that $Z$ does not fix any non-zero nilpotent element; thus it fixes a non-zero semisimple element and is contained in a maximal rank subgroup $M'$, again by \cite[Lemma~1.3]{MR1048074}. This part is rather technical and requires the full classification of stabilisers of nilpotent elements and their structure, as found in \cite{MR2883501}. It is however then possible to identify that $M' = M$ and show that $Z$ is conjugate to $H$.

\subsection{\texorpdfstring{$G$-conjugacy}{G-conjugacy}}

Once we know that two isomorphic candidates $H_1$ and $H_2$ are $G$-irr, we must check whether they are $G$-conjugate. One easy test is to check whether their composition factors on various $G$-modules agree. If so then it turns out that, with a single exception, the two candidates are in fact $G$-conjugate. The exception occurs when $G = E_8$, $p \neq 3$ and $H_1$ and $H_2$ are simple of type $A_2$, diagonally embedded in $A_2^2 \subset D_4^2 \subset D_8 \subset E_8$ via $(10,10^{[r]})$ and $(10,01^{[r]})$ respectively, with $r \neq 0$. For more detail see \cite[Corollary~1]{ThomasIrreducible} and its proof. 

If $H_1$ and $H_2$ are not simple then it is usually straightforward to show they are conjugate when they have the same composition factors on $\text{Lie}(G)$, by considering the centraliser of one of the simple factors. 

\begin{example}Let $G = E_8$, let $M_1$ be the maximal-rank subgroup of type $A_1 E_7$ and let $M_2$ be the maximal-rank subgroup of type $D_8$. Take $H_1 = A_1^2 D_6 \subset M_1$ and $H_2 = A_1^2 D_6 \subset D_8 = M_2$. Then $H_1$ and $H_2$ are $G$-irr and $G$-conjugate. Indeed, taking $Y$ to be one of the $A_1$ factors of $H_2$, we have $H_2 \subset Y C_G(Y) = M_1$, by appealing to \cite[p.333, Table~2]{MR1274094}. As $H_1$ is the only subgroup of type $A_1^2 D_6$ contained in $M_1$ up to conjugacy, it is conjugate to $H_2$.
\end{example}

When the candidates are simple this process can be slightly more involved. Often one of the candidates turns out to be the connected centraliser in $M$ or $G$ of an involution, or of an element of order $3$.

\begin{example} Again let $G = E_8$ and $p \neq 2$. Take $H = B_4 \subset A_8$, with $H$ acting irreducibly on the natural $9$-dimensional module for $A_8$. Then $H$ is the centraliser in $G$ of an involution $t$ in the disconnected subgroup $A_8.2$. By calculating the trace of this involution on the adjoint module for $G$ and using \cite[Proposition~1.2]{MR1717629}, we find that $C_G(t) = D_8$ and hence $H \subset D_8$. Similar calculations are carried out in \cite[pp.~56--68]{MR1329942}.
\end{example}     

\subsection{Main results}

We now present results classifying the $G$-cr semisimple subgroups of exceptional algebraic groups. When $p$ is large enough that all subgroups of a given type are $G$-cr, the simple $G$-cr subgroups were classified in \cite{MR1329942,LawTest1}. The $G$-irr subgroups of type $A_1$ were studied for $G$ of exceptional type except $E_8$ in \cite{MR2707891}, and $G$-irr subgroups of $G = G_2, F_4$ are classified in \cite{MR2604850}, \cite{SteF4}, respectively. The reductions in \S\ref{sec:gcralg} now allow us to concentrate on semisimple, $G$-irr subgroups.

\begin{theorem} \label{t:Girr}
Let $G$ be a simple algebraic group of exceptional type and let $H$ be a $G$-irr connected subgroup of $G$. Then $H$ is $\Aut(G)$-conjugate to exactly one subgroup in Tables \cite[\S 11, Tables 1--5]{ThomasIrreducible} and each subgroup in the tables is $G$-irr.
\end{theorem}
This was proved in a sequence of papers \cite{Tho1,Tho2,ThomasIrreducible}. The tables are lengthy and we do not reproduce them here. The last of these papers, specifically \cite[\S 11, Tables 1A--5A]{ThomasIrreducible}, describes the poset structure of $\mathrm{ConIrr}(G)$ and gives a detailed explanation of the tables. The tables also provide the composition factors of each subgroup in $\mathrm{ConIrr}(G)$ on both the minimal and adjoint module.

\begin{example} \label{eg:G2irr}
When $G=G_2$, the reductive maximal connected subgroups of $G$ are: \[A_1 \tilde{A}_1, \ \ A_2, \ \  \tilde{A}_2 \ (p=3), \quad\text{and}\quad  A_1 \ (p \geq 7),\] where $\tilde{A}_1$ and $\tilde{A}_2$ denote subgroups whose roots are short roots of $G$. A group of type $A_1$ has no proper irreducible connected subgroups, so this requires no further consideration. 

Take $M_1 = A_1 \tilde{A}_1$. Since the factors have the same type, there are diagonal $M$-irr connected subgroups of type $A_1$. As in Example~\ref{ex:diagonalA1A1}, these are determined by non-negative integers $r,s$ with $rs=0$ and we write $H^{r,s}$ for such a subgroup. So $\mathrm{ConIrr}(M_1) = \{ H^{r,s}\mid rs=0\}$. We now need to decide whether the members of $\mathrm{ConIrr}(M_1)$ are $G$-irr. If $(r,s)\neq(0,0)$ or $p > 3$ then $H^{r,s}$ acts on $\Lie(G)$ without trivial composition factors, hence it is $G$-irr. When $p=3$, the composition factors of $H^{0,0}$ on $\text{Lie}(G)$ do not match those of a Levi subgroup and so $H^{0,0}$ is also $G$-irr. When $p=2$, however, $H^{0,0}$ is contained in an $A_1$-parabolic subgroup and is non-$G$-cr (appearing in Theorem~\ref{t:nonGcrG2}). To see that $X := H^{0,0}$ is contained in a proper parabolic subgroup it suffices to demonstrate that $X$ stabilises a $1$-space on the irreducible $6$-dimensional module $L_{G}(10)$, since by \cite[Theorem~B]{MR1316858}, $G$ is transitive on $1$-spaces of $L_{G}(10)$ and the stabiliser of such a $1$-space is a long root parabolic subgroup. It remains to show that $X$ is not conjugate to $\D(L)$ for some Levi subgroup $L$ (up to conjugacy, these are the two simple factors of $M$). This can be done by calculating the action of $X$ on $L_G(10)$, which is $T(2) + 2$ and then comparing it with the actions of the two subgroups $\D(L)$ on $L_{G}(10)$, which are $1^2 + 0^2$ and $1^2 + 2$, respectively. 

Now let $M_2 = A_2$. Applying Proposition~\ref{prop:irredclass}, we find a single candidate, $H$, which has type $A_1$ with $p \neq 2$, and this acts irreducibly on the adjoint $3$-dimensional module $L(2)$. We must check that $H$ is $G$-irr. In fact, one can show $H$ is conjugate to $H^{0,0}$ and thus $G$-irr. To see this, note that $N_{G}(M_2) = M_2\langle t\rangle$, with $t$ an involution inducing a graph automorphism on $M_2$ \cite[Table~4.3.1]{MR1490581}. As $L(2)$ is self-dual, one concludes that $H$ centralises $t$ and hence $H\subset C_G(t)=M_1$. The only subgroups of type $A_1$ in $M$ are its two simple factors and the subgroups $H^{r,s}$. Since $p \neq 2$, the composition factors of the action of these subgroups on $V_G(10)$ distinguish them and we conclude that $H$ is conjugate to $H^{0,0}$.

The same method applies to $\tilde{A}_2$ when $p = 3$ and one finds a single candidate subgroup $H$ of type $A_1$, which turns out to be $G$-irr and conjugate to $H^{1,0}$.

We present this classification in Figure \ref{picg2subs}, with a straight line depicting containment. This gives a small flavour of the additional information in \cite{ThomasIrreducible}. 
\begin{figure}[htbp]
\centering
\begin{tikzpicture}[yscale=0.95]
\node at (0,0) {};
\node at (\textwidth,0) {};
\node at (0,-4) {};
\node at (\textwidth,-4) {};
\draw  (0.5\textwidth,0.7) node[anchor=north,align=center]{$G_2$}; %$G_2 = G_2(\#\gtwo{0})$
\draw (0.1\textwidth,-2) node[align=center]{$A_2$ \\ ${}$};
\draw (0.315\textwidth,-2) node[anchor=west,align=center]{$A_1 \tilde{A}_1$ \\ ${}$};
\draw (0.63\textwidth,-2) node[anchor=west,align=center]{$\tilde{A}_2 $ \\ $(p=3)$};
\draw (0.87\textwidth,-2) node[anchor=west,align=center]{$A_1$ \\ $(p \geq 7)$};
\draw (0.1\textwidth,-4.5) node[align=center]{$A_1 = H^{0,0}$ \\ $(p\neq2)$};
\draw (0.63\textwidth,-4.5) node[align=center,anchor=west]{$A_1 = H^{1,0}$ \\ $(p = 3)$};
\draw (0.2\textwidth,-4.8) node[align=center,anchor=west]{$A_1 = H^{r,s}$ \\ ($(r,s) \neq (0,0)$; \\ if $p=3$ then $(r,s) \neq (1,0)$)};
\draw (0.1\textwidth,-1.4)--(0.5\textwidth,0);
\draw (0.35\textwidth,-1.4)--(0.5\textwidth,0);
\draw (0.65\textwidth,-1.4)--(0.5\textwidth,0);
\draw (0.9\textwidth,-1.4)--(0.5\textwidth,0);
\draw (0.1\textwidth,-2.2)--(0.1\textwidth,-3.9);
\draw (0.35\textwidth,-2.2)--(0.1\textwidth,-3.9);
\draw (0.35\textwidth,-2.2)--(0.35\textwidth,-3.9);
\draw (0.35\textwidth,-2.2)--(0.67\textwidth,-3.9);
\draw (0.67\textwidth,-2.5)--(0.67\textwidth,-3.9);
\end{tikzpicture}
\caption{The poset of $G_2$-irr connected subgroups. \label{picg2subs}} 
\end{figure}
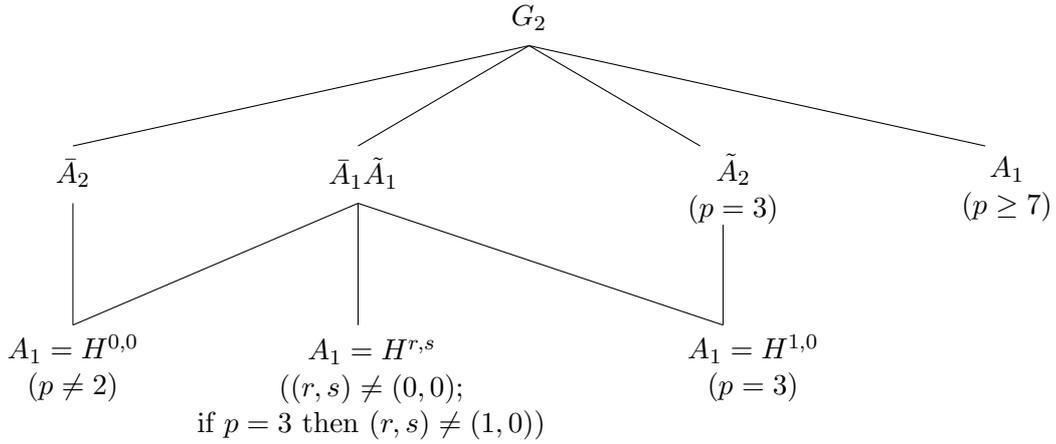
\end{example}

\begin{remark}\leavevmode
Further information about the $G$-reducible semisimple $G$-cr subgroups can be found in \cite{Litterick2018}. In particular, the semisimple $G$-reducible $G$-cr connected subgroups of $G=F_4$ are classified and written down explicitly in \cite[\S 6.4]{Litterick2018a}, correcting \cite[Corollary~5]{SteF4}. For each reducible $G$-cr subgroup $H$, the socle series of the action of $H$ on $\text{Lie}(G)$ is computed together with the centraliser $C_G(H)^\circ$ (which is another $G$-cr subgroup by Theorem~\ref{thm:BMRClifford}). 
\end{remark}

\section{Non-\texorpdfstring{$G$-}{}completely reducible subgroups} \label{s:nonGcr}

The method in \S\ref{s:nonGcrstrat} has been applied to classify the non-$G$-cr semisimple subgroups of exceptional algebraic groups $G$ in many cases. The following is a more precise version of \cite[Theorem~1]{MR1329942}. It generalises Theorem \ref{thm:gcr-red} for exceptional groups by ruling out certain non-$G$-cr subgroups of certain types even when $p \leq \text{rank}(G)$.  

\begin{theorem} \label{t:subtypes}
Let $G$ be a simple algebraic group of exceptional type in characteristic $p$. Let $X$ be an irreducible root system and let $N(X,G)$ be the set of primes defined by Table~\ref{tab:newimrn}, e.g.~$N(B_2,E_8)=\{2,5\}$.

If $H$ is a connected reductive subgroup of $G$ and $p \notin N(X,G)$ whenever $H$ has a simple factor of type $X$, then $H$ is $G$-cr. Conversely, whenever $p \in N(X,G)$, there exists a non-$G$-cr simple subgroup $H$ of type $X$.

\end{theorem}

\begin{table}[htbp] \centering
\begin{tabular}{r|rrrrr}
&$G=E_8$&$E_7$&$E_6$ & $F_4$ & $G_2$\\ \hline
$X=A_1$ & $\leq 7$ & $\leq 7$ & $\leq 5$ & $\leq 3$ & $2$ \rule{0pt}{2.4ex}\\
$A_2$ & $\leq 3$ & $\leq 3$ & $\leq 3$ & $3$\\
$B_2$ & $5\ 2$ & $2$ & $2$ & $2$\\
$G_2$ & $7\ 3\ 2$ & $7\ 2$ & & \\
$A_3$ & $2$ & $2$ & & \\
$B_3$ & $2$ & $2$ & $2$ & $2$\\
$C_3$ & $3$ & & & \\
$B_4$ & $2$ & & \\
$C_4, D_4$ & $2$ & $2$ & 
\end{tabular}
\caption{Values of $N(X,G)$.} \label{tab:newimrn}
\end{table}

\begin{remark} This corrects \cite[Theorem~1]{Stewart21072013}, which had claimed the existence of non-$G$-cr subgroups $H$ of type $G_2$ when $p=2$ and $G$ is of type $F_4$ or $E_6$. We discuss this further in Example \ref{ex:G2inF4E6}.
\end{remark}

The result ``$p \not \in N(G,X) \Rightarrow$ all simple subgroups of type $X$ are $G$-cr'' is largely proven in \cite{MR1329942}. The strategy is to show that for each parabolic subgroup $P = QL$ and each $L$-irr subgroup $\bar{X} < L$, the levels of $Q$ restricted to $\bar{X}$ have trivial $1$-cohomology. In \emph{loc.~cit.} and \cite{SteF4} this is accomplished by determining all modules $V$ with $\opH^1(\bar{X},V) \neq 0$ such that $\dim V$ is small enough for $V$ to potentially appear as an $\D(L)$-composition factor in the filtration of $Q$. For $G$ of type $E_8$, the largest dimension of such a $\D(L)$-composition factor is $64$, occurring when $\D(L)$ is of type $D_7$, which makes this a tractable problem. (See \cite[Lemma~5.1]{Stewart21072013} for more information on the upper bounds on $\dim V$.) 

\begin{example}
We use results from \S\ref{sec:BMRresults} to exhibit a non-$G$-cr subgroup of type $A_1$ in $G = E_6$ when $p = 5$. Here, $G$ has a Levi subgroup $L$ with $\D(L)$ of type $A_5$. It follows from Theorem \ref{thm:BMRcentraliser} that any non-$L$-cr subgroup is non-$G$-cr. So it suffices to find a non-$L$-cr subgroup of type $A_1$, which is equivalent to finding an indecomposable, reducible $6$-dimensional module for $\text{SL}_2$. To that end, if $V\cong L(1)$ is the natural module for $\text{SL}_2$ then the $p$-th symmetric power $S^p(V)$ is indecomposable of dimension $p+1$ with two composition factors, $L(p)$ and $L(0)$.
\end{example}

The following is a delicate case of the proof that every simple subgroup of type $X$ is $G$-cr when $p \not\in N(X,G)$; this corrects an error in \cite{SteF4}.

\begin{example} \label{ex:G2inF4E6}
Let $G = F_4$ and $p = 2$. We show that every subgroup of type $G_2$ is $G$-cr, and note that a similar but easier argument applies for $G = E_6$. The only proper Levi subgroups with a subgroup of type $G_2$ are those with derived subgroup $B_3$ or $C_3$. Write $P = QL$ where $\D(L) = B_3$ and let $H$ be the $L$-irr subgroup of type $G_2$. In \cite[Lemma~4.4.3]{SteF4} it is claimed that $\opH^1(G_2,Q)$ is $1$-dimensional, which in turn relies on the claim that an element of $(Q/Q_2)^H$ lifts to an element of $Q^H$. However, this latter claim is false. To see this, note that $Q^H \subset C := C_G(H)^\circ$. Since $H$ is $G$-cr, so is the subgroup $C$ by Theorem \ref{thm:BMRClifford}~\ref{gcr-props-ii}, hence this is reductive and moreover has rank $1$ since $Z(L)$ is $1$-dimensional, cf.~Remark~\ref{rem:minLevis}.

By the Borel--de-Siebenthal algorithm, $H\subset \D(L)$ centralises the subgroup $\tilde A_1$ of $G$. Thus $C$ has type $A_1$, and has no $2$-dimensional unipotent subgroup, so $\dim Q^H \le 1$. By restricting the $\D(L)$-action on $Q_2$ to $H$, we find that $Q_2 \downarrow H \cong V(1,0) = L(1,0)/L(0,0)$, where $V(1,0)$ is the Weyl module of high weight $(1,0)$ with the trivial module $L(0,0)$ in its socle. Thus $Q_2^H\cong k$. As $Q_2^H\subseteq Q^H$, we have equality by comparing dimensions. Therefore, no non-trivial element of $(Q/Q_2)^H$ lifts to an element of $Q^H$. From the exact sequence \eqref{eq:long}, it then follows that the map $\opH^1(H,Q_2) \to \opH^1(H,Q)$, which is surjective since $\opH^1(H,Q/Q_2) = 0$, is the zero map. Thus $\opH^1(H,Q)=0$, so $H$ gives rise to no non-$G$-cr subgroups in $P$. Applying the graph morphism of $G$ now allows one to conclude that the $C_3$-parabolic also contains no non-$G$-cr subgroups of type $G_2$.
\end{example}

Moving on, let us assume $p\in N(X,G)$. Then there exist non-$G$-cr semisimple subgroups with a factor of type $X$, and we wish to classify them all. The case $G = G_2$ was settled by the second author in \cite{MR2604850}. Since the classification is unusually short, we present it here. Recall that $G_2$ has semisimple maximal connected subgroups of type $A_1 \tilde{A}_1$ and $A_2$ (see Theorem~\ref{t:maximalexcep}).

\begin{theorem} \label{t:nonGcrG2}
Let $G$ be a simple algebraic group of type $G_2$ in characteristic $p$, and let $H$ be a non-$G$-cr semisimple subgroup. Then $p=2$, $H$ has type $A_1$ and is $G$-conjugate to precisely one of $Z_1$ and $Z_2$ below.
\begin{enumerate}[label=\normalfont(\roman*)]
\item $Z_1 \subset A_1\tilde{A}_1$ embedded diagonally via $x\mapsto (x,x)$;
\item $Z_2 \subset A_2$ embedded via $V(2)\cong L(2)/L(0)$.
\end{enumerate}
\end{theorem}

This theorem exhibits an interesting feature: every non-$G$-cr subgroup has a proper reductive overgroup in $G$. This is not true in general, and we consider semisimple subgroups with no proper reductive overgroups further in \S\ref{sec:noredovergps}.

For the next result, recall that a prime $p$ is called \emph{good for $G$} if $G$ has type $G_2, F_4, E_6, E_7$ and $p > 3$, or if $G$ has type $E_8$ and $p > 5$; otherwise $p$ is called \emph{bad} for $G$. The non-$G$-cr semisimple subgroups of $F_4$ were extensively studied by the second author in \cite{SteF4}; however the paper contains a number of errors. These are systemically dealt with in \cite{GanTho}, from which we distil a headline result, presented together with the classification of non-$G$-cr semisimple subgroups in good characteristic due to the first and third authors \cite{Litterick2018}. 

\begin{theorem} \label{t:nonGcrF4E8}
Let $G$ be a simple algebraic group of exceptional type in characteristic $p > 0$ and let $H$ be a non-$G$-cr semisimple subgroup of $G$. Then one of the following holds: 
\begin{enumerate}[label=\normalfont(\roman*)]
\item $(G,p) = (G_2,2)$ and $H$ is a subgroup $Z_1$ or $Z_2$ from Theorem~\ref{t:nonGcrG2};
\item $(G,p) = (F_4,3)$ and $H$ has type $A_2$, $A_1A_1$ or ${}^\star A_1$; 
\item $(G,p) = (F_4,2)$ and $H$ has type $B_3$, $B_2$, ${}^\star A_1 B_2$, ${}^\star A_1 A_2$ or ${}^\star A_1^n$ with $n \leq 3$;
\item $(G,p) = (E_6,5)$ and $H$ has type $A_1^2$ or $A_1$;
\item $(G,p) = (E_7,5)$ and $H$ has type $A_2A_1$, $A_1^2$ or $A_1$;
\item $(G,p) = (E_7,7)$ and $H$ has type $G_2$ or $A_1$;
\item $(G,p) = (E_8,7)$ and $H$ has type $A_1 G_2$, $A_1^2$ or $A_1$;
\item $G$ has type $E_6$, $E_7$ or $E_8$ and $p$ is bad for $G$. \label{case-eight}
\end{enumerate} 
Conversely, there is a non-$G$-cr semisimple subgroup of each type listed and infinitely many conjugacy classes for those marked ${}^\star$. 
\end{theorem}

Case \ref{case-eight} remains the subject of ongoing work of the first and third authors. 

\begin{remark}\leavevmode
When $p$ is good for $G$, a non-$G$-cr semisimple subgroup is $G$-conjugate to precisely one subgroup in \cite[Tables~11--17]{Litterick2018} and conversely, each subgroup in those tables is non-$G$-cr. Furthermore, \cite{Litterick2018} also provides the connected centraliser of each subgroup and the action on minimal and adjoint modules. 
\end{remark}

\subsection{Semisimple subgroups with no proper reductive overgroups} \label{sec:noredovergps}

To describe the poset of connected reductive groups, one needs to describe the maximal elements. If $H$ is one such, then either: $H$ is maximal amongst all connected subgroups; or $H$ has a non-trivial central torus $S$, so that $H = C_G(S)$ is a Levi subgroup; or $H$ is semisimple and non-$G$-cr. 

\begin{example}
Let $(G,p) = (E_7,7)$. We exhibit a non-$G$-cr subgroup of type $G_2$ which is maximal amongst proper reductive subgroups of $G$. In fact, this is unique up to conjugacy (see \cite[\S6.1]{Litterick2018}). Let $P = QL$ be a parabolic subgroup of $G$, where the derived subgroup $\D(L)$ has type $E_6$. This has a maximal subgroup of type $F_4$ which itself has a maximal subgroup $\bar{H}$ of type $G_2$ when $p = 7$ (see Theorem \ref{t:maximalexcep} and \cite[Theorem 1(c)]{MR1100666}). The subgroup $\bar{H}$ turns out to be $\D(L)$-irr (\cite[Theorem~1]{Tho1}). Now the unipotent radical $Q$ is abelian, a $27$-dimensional module for $\D(L)$, which implies that $Q$ is isomorphic to either $L_{E_6}(\lambda_1)$ or its dual $L_{E_6}(\lambda_6)$. Moreover, $Q \downarrow \bar H = L(20) \oplus L(00)$, and $\opH^1(\bar{H},L(20))$ is $1$-dimensional.\footnote{This follows, for example, by a dimension-shifting argument \cite[II.2.1(4)]{Jan03} with the induced module $\opH^0(20) \cong L_{G_2}(00)/L_{G_2}(20)$.} This implies that $\opH^1(\bar{H},Q) \cong k$. Furthermore, the torus $Z(L)^\circ$ acts by scalars on $Q$, inducing conjugacy amongst the non-zero elements of $\opH^1(\bar{H},Q)$ and it follows that there is a unique $G$-conjugacy class $[H]$ of non-$G$-cr subgroups of type $G_2$ complementing $Q$ in $Q\bar H$.

In \cite[\S 10]{Litterick2018} it is proved that if $V$ is the $56$-dimensional module for $E_7$ then $V \downarrow H = T(20) \oplus T(20)$ where $T(20)=L(00)/L(20)/L(00)$ is tilting, of high weight $20$. It follows that any reductive overgroup of $H$ acts on $V$ either indecomposably or with two indecomposable summands of dimension $\dim T(20) = 28$. Inspecting the maximal subgroups of $G$ and their actions on $V$, we see the only plausible maximal reductive connected overgroup has type $A_7$. But the only non-trivial $8$-dimensional $G_2$ modules are Frobenius twists of $L(10) \oplus L(00)$, and so any $G_2$ subgroup of $A_7$ is contained in a Levi subgroup $A_6$. But these act on $V$ with four indecomposable summands, which rules out $A_7$ as a reductive overgroup of $H$.
\end{example}

There are several more instances of non-$G$-cr semisimple subgroups with no proper connected reductive overgroups, which are thus maximal amongst connected reductive subgroups of $G$. The following partial result begins to extend Theorem \ref{t:maximalexcep} towards describing the classes of maximal connected reductive subgroups. It can be deduced by combining Theorem \ref{t:maximalexcep} and the main results of the references given for Theorem \ref{t:nonGcrF4E8}. Out of interest, we note here a geometric interpretation: Since reductive subgroups of $G$ correspond to affine coset spaces \cite{MR437549}, a subgroup $H$ which is maximal amongst reductive subgroups corresponds to $G/H$ being minimal (with respect to $G$-quotients) amongst non-trivial affine homogeneous $G$-spaces. Then $H$ being maximal amongst \emph{connected} reductive subgroups means $G/H$ is minimal up to a \emph{finite-sheeted} quotient.

\begin{theorem} \label{t:maximalred}
Let $G$ be a simple algebraic group of exceptional type in characteristic $p$, and let $M$ be maximal amongst connected reductive subgroups of $G$. Then one of the following holds: 
\begin{enumerate}[label=\normalfont(\roman*)]
\item $M$ is maximal amongst connected subgroups and is $G$-conjugate to precisely one subgroup in Table \ref{tab:maximalexcep}; 
\item $M$ is a Levi subgroup, with $(G,\D(M)) = (E_6,D_5)$ or $(E_7,E_6)$;
\item $(G,p) = (F_4,3)$ and $M$ is non-$G$-cr of type $A_1$; 
\item $(G,p) = (F_4,2)$ and $M$ is non-$G$-cr of type $B_2$ or $A_1A_1$; 
\item $(G,p) = (E_7,7)$ and $M$ is non-$G$-cr of type $G_2$ (unique up to conjugacy);  
\item $G$ has type $E_{6}$, $E_7$ or $E_8$ and $p$ is bad for $G$. \label{maxred-case-vi}
\end{enumerate}                                                                                
\end{theorem}

\begin{remark}\leavevmode
Case \ref{maxred-case-vi} will contain many conjugacy classes of subgroups. An interesting example is a class of subgroups of type $A_2$ in $G = E_8$ which act on the adjoint module with indecomposable summands of dimension $240$ and $8$, which already precludes its containment in a proper connected reductive subgroup. 
\end{remark}

\section{Further directions and related problems} \label{sec:further}

\subsection{Hereditary subgroups} \label{s:hereditary}

Recall from \cite{MR2167207} and \S\ref{ss:criteria} that one of Serre's reasons to formalise $G$-complete reducibility was for the study of converse theorems. An archetypal question \cite[\S5.3, Remarque]{MR2167207} is: 

\textit{When does the exterior square of a module being semisimple imply that the original module is semisimple?} 

This can be translated into $G$-complete reducibility as follows. Let $H$ be a semisimple algebraic group, $V$ an $H$-module and suppose that $\wedge^2(V)$ is semisimple. The action of $\SL(V)$ on $\wedge^2(V)$ furnishes inclusions $\bar H\subseteq M\subseteq G := \text{SL}(\wedge^2(V))$, where $\bar H$ and $M$ are the images in $G$ of the groups $H$ and $\SL(V)$, respectively. Then the question above asks whether $H$ being $G$-cr implies that $H$ is $M$-cr.

\begin{defn} Let $M$ be a connected subgroup of a reductive group $G$. We define $M$ to be \emph{$G$-ascending hereditary} ($G$-ah) if, for all connected subgroups $H$ of $M$, if $H$ is $M$-cr then $H$ is $G$-cr. And $M$ is defined to be \emph{$G$-descending hereditary} ($G$-dh) if, for all connected subgroups $H$ of $M$, if $H$ is $G$-cr then $H$ is $M$-cr. We define $M$ to be \emph{$G$-hereditary} if it is both $G$-ah and $G$-dh. 
\end{defn}

Let $G$ be a simple algebraic group with subgroups $H \subset M$. Theorem~\ref{thm:BMRcentraliser} says that centralisers of linearly reductive subgroups of $G$, such as Levi subgroups, are $G$-hereditary. Theorem~\ref{thm:sep-red} gives criteria for a subgroup $M$ to be $G$-dh. We have also seen examples of non-hereditary subgroups. Theorem~\ref{t:nonGcrG2} shows that in $G = G_2$, the subgroup $H = Z_1$ is $M$-irreducible for $M = A_1\tilde{A}_1$, but that $H$ is non-$G$-cr; so $M$ is non-$G$-ah. The following example shows that non-$G$-dh subgroups also exist:   

\begin{example} \label{ex:G2inF4again}
Let $G = F_4$ and $p=2$. Then $G$ contains subgroups $C_4$ and $B_4$, which respectively contain subgroups $\tilde{D}_4$ and $D_4$, and these in turn respectively contain simple subgroups $H_1$ and $H_2$ of type $G_2$. The simple module $L_{C_4}(\lambda_1) \downarrow H_1 = T(10)$, and the Weyl module $V_{B_4}(\lambda_1) \downarrow H_2 = T(10) \oplus 00$. These subgroups are swapped by the exceptional isogeny of $G$. By Proposition~\ref{prop:irredclass}, $H_1$ is non-$C_4$-cr and thus $H_2$ is non-$B_4$-cr. However, in Example~\ref{ex:G2inF4E6} we saw that every subgroup of type $G_2$ is $G$-cr. So $H_1$ and $H_2$ are examples of $G$-cr subgroups that are non-$M$-cr in some reductive maximal subgroup $M$ of $G$. Therefore, $M$ is not $G$-dh. 
\end{example}

Given the plethora of results (\S\ref{ss:criteria}, \cite{MR2178661} and elsewhere) which guarantee $G$-hereditary behaviour, it is natural to ask how often such behaviour fails. Ongoing work of the first and third authors seeks more precise results in this vein. For instance, it is very uncommon for a subgroup $H$ in a reductive pair $(G,H)$ (cf.~p.~\pageref{par:reductive-pair}) to be non-$G$-dh, in fact based on empirical evidence the first and third authors speculate:
\begin{conjecture} 
Let $(G,H)$ be a reductive pair of algebraic groups in characteristic $p$. If $H$ is non-$G$-dh then $p=2$. 
\end{conjecture}

\subsection{Unipotent elements in exceptional algebraic groups} \label{ss:unipclasses}

Much effort has been spent studying unipotent elements in algebraic groups. It is even non-trivial to show that there are finitely many unipotent classes in the exceptional groups. There is an extensive literature on this and the book \cite{MR2883501} contains a comprehensive treatment. We mention those results most closely related to subgroup structure.

For many applications it is useful to understand how an embedding $H \to G$ fuses classes of unipotent elements. When $H$ is a reductive maximal connected subgroup this has been completely determined by Lawther in \cite{MR2526390} (supplemented by \cite{MR4375734} for the newly-discovered maximal subgroup of type $F_4$ in $E_8$ when $p=3$).   

One can flip this question and ask: given a unipotent element, what are its overgroups? For example, {\it regular unipotent elements} are those whose centralisers have the smallest possible dimension (the rank of $G$), and these are all $G$-conjugate. Overgroups of regular unipotent elements have been heavily studied, first of all by Saxl--Seitz in \cite{MR1438641}, classifying the maximal positive-dimensional reductive subgroups containing a regular unipotent element. Extending this to all positive-dimensional reductive subgroups containing a regular unipotent element is difficult. Suppose that $H$ is a subgroup of $G$ containing a regular unipotent element. Testerman--Zalesski \cite[Thm.~1.2]{MR2988707} provided a full classification of the connected reductive subgroups containing a regular unipotent element. An important ingredient is to show that if $H$ is connected then it is $G$-irr. If instead $H$ is only assumed to be positive-dimensional then Malle--Testerman \cite[Theorem~1]{MR4312324} show that either $H$ is $G$-irr or $H^\circ$ is a torus; the latter case does in fact occur. Indicative of the narrative in Part I, Bate--Martin--R\"{o}hrle in \cite{MR4386351} were able to produce a uniform proof of \cite[Theorem~1.2]{MR2988707} and \cite[Theorem~1]{MR4312324}, and further generalisations to disconnected groups, finite groups of Lie type and Lie algebras, without intricate case-by-case considerations. A key ingredient in their proof was the observation of Steinberg that regular unipotent elements normalise a unique Borel subgroup of $G$.  

\subsection{Finite subgroups and groups of Lie type} \label{ss:maxfinsub}

Historically, one of the main motivations for studying maximal and then reductive subgroups of exceptional algebraic groups was to deduce results for the exceptional finite groups of Lie type, see \cite{MR1994964} for the work up to the early 2000s. A related problem is to understand finite subgroups of the exceptional algebraic groups, whose study cannot employ any of the techniques requiring connectedness of the subgroup. In attempting to use the strategy of \S\ref{sec:strat} for a simple algebraic group $G$ of exceptional type, the primary difficulty is in classifying $G$-irr subgroups, since ad-hoc methods are required in place of uniform statements about representations of reductive groups. A result of Borovik \cite[Theorem 1]{MR1066315} quickly reduces one to studying almost-simple finite subgroups, and the isomorphism types of finite simple subgroups have been enumerated by Cohen, Griess, Serre, Wales and others (e.g.~\cite{MR1320515,MR1066315,MR1065645}) over the complex numbers using character-theoretic methods, and by Liebeck and Seitz \cite{MR1717629} in positive characteristic.

In positive characteristic, it is essential to understand \emph{generic subgroups}, i.e.~finite groups of Lie type $H(q)$ with embeddings $H(q) \to G$ which factor through an inclusion $H \to G$ of algebraic groups. The main result in this direction is due to Liebeck and Seitz \cite{MR1458329}, giving an explicit bound on $q$ (usually $q > 9$) ensuring that \emph{all} embeddings $H(q) \to G$ arise in this fashion.

Further progress on \emph{non-generic} subgroups has been made by the first author \cite{MR3803556} by comparing Brauer characters of finite simple groups with the Brauer traces of elements of exceptional algebraic groups $G$ on low-dimensional modules, which limits the composition factors of simple groups acting on these modules. This carries sufficient information to rule out $G$-irr subgroups, for instance a subgroup fixing a vector on $\Lie(G)$ lies in the corresponding stabiliser, which is often parabolic. One can also rule out the existence of non-$G$-cr subgroups: For instance, if a subgroup is contained in a parabolic subgroup of some algebraic group $G$, then it normalises the unipotent radical, and the Lie algebra of this is a submodule of $\Lie(G)$. So if the subgroup has no composition factors on $\Lie(G)$ with non-zero cohomology, the subgroup cannot have non-zero cohomology in its action on the radical, so is $G$-cr. Craven \cite{MR3694291} has extended these techniques to also bring in Jordan block sizes of unipotent elements of finite groups acting on the relevant modules, deriving still stronger conditions and ruling out further subgroup types.

For \emph{maximal} subgroups of finite groups of Lie type, there has been considerable recent progress. The maximal subgroups of ${}^2B_2(q)$, ${}^2F_4(q)$, ${}^3D_4(q)$, ${}^2G_2(q)$ and $G_2(q)$ have been classified by Cooperstein, Kleidman, Malle, Suzuki \cite{MR618376,MR937609,MR1106340}; leaving $F_4(q)$, $E_6^\epsilon(q)$, $E_7(q)$ and $E_8(q)$ to consider. The maximal subgroups of these groups have been completely classified for some very small $q$ (e.g.~$E_7(2)$ in \cite{MR3361643}) and there has been considerable progress by Magaard \cite{MR2638705} in the case $F_4(q)$ and Aschbacher \cite{MR892190,MR928524,MR986684,MR1054997} in the case $E_6(q)$. The most recent work has led to a complete classification of the maximal subgroups of $F_4(q)$, $E_6(q)$ and ${}^2E_6(q)$ \cite{Craven21} and almost a complete classification for $E_7(q)$ \cite{Craven22b}. This all builds on previous work \cite{Craven22,Craven22c,MR3694291}. The maximal subgroups of $E_8(q)$ are also a work in progress by Craven. 

We round off our discussion here by mentioning one more way in which $G$-complete reducibility applies to finite groups of Lie type: Namely, through \emph{optimality}. The details are somewhat technical (cf.~\cite[Def.~5.17]{MR3042598}), but a key point is: A non-$G$-cr subgroup $H$ is contained in a \emph{canonical} parabolic subgroup of $G$, which is normalised by all automorphisms of $G$ which normalise $H$. Applying this to a Frobenius endomorphism $F$ of $G$, if $H$ is a non-$G$-cr finite subgroup of some group of Lie type $G(q) = G^F$, we get a method of constructing subgroups in between $H$ and $G(q)$, namely, $F$-fixed points of the corresponding canonical parabolic subgroup (cf.~\cite[Proposition~2.2]{MR2145743} and \cite[\S 4.1--4.2]{MR3803556}).

\subsection{Variations of complete reducibility} \label{s:variations}

Serre's definition \ref{def:gcr} admits generalisations in various directions. If $G$ is equipped with a Frobenius endomorphism $F$ and one considers only $F$-stable parabolic and Levi subgroups, one arrives at so called `$F$-complete reducibility' \cite{MR2783312}. In another direction, as mentioned in the introduction, $G$-complete reducibility generalises to disconnected reductive groups, if one is willing to work instead with R-parabolic and R-Levi subgroups, and many results in $G$-complete reducibility generalise at once (cf.~\cite[\S 6]{MR2178661}). The resulting geometric invariant theory is in fact the natural setting in which to derive the most general results, only a handful of which have been mentioned in the present article. Since one is now working with collections of morphisms $\Gm \to G$, restricting these morphisms to land in a subgroup $K$ yields yet another generalisation, `$G$-complete reducibility with respect to $K$', and once again many natural results extend immediately \cite{MR4108921,MR4077197}.

Ultimately, one can view complete reducibility as a property of the \emph{spherical building} of $G$ \cite{MR2167207}. Here, opposite parabolics are opposite simplices, and non-$G$-cr subgroups correspond to contractible subcomplexes; omitting all details, we simply mention that recasting the above results in terms of the building allows one to unite the various generalisations above and derive yet stronger statements, e.g.~\cite{MR4440711}, and even extend to Euclidean buildings, Kac--Moody groups and other settings, see for instance \cite{MR2982240} and \cite[\S 4.3]{MR2499773}.

\subsection{Structure of the Lie algebra of exceptional algebraic groups} \label{ss:maxliealg}
Through the exponential and logarithm maps, classifying maximal connected subgroups of a complex algebraic group $G$ is equivalent to classifying maximal subalgebras of $\Lie(G)$ and indeed Dynkin's original work is set in this context. It was noticed by Chevalley that any complex finite-dimensional simple Lie algebra $\g$ has a $\Z$-basis. This means there is an integral form $\g_\Z$ from which one may build a Lie algebra $\g_R:=\g_\Z\otimes_\Z R$ over any commutative ring $R$. In particular, we may take $R=k$ for $k$ an algebraically closed field of characteristic $p>0$. There is typically more than one $\Z$-form available, leading to non-isomorphic Lie algebras over $k$; Chevalley's recipe gives the simply-connected form, i.e.~$\g_k\cong\Lie(G)$, where $G$ is the simply-connected algebraic group over $k$ of the same type (\cite[Chap.~VII]{MR499562}). 

It is natural to ask about maximal subalgebras of $\g=\g_k$ and more generally its subalgebra structure and their conjugacy under the adjoint action of $G$. Motivation for this question also arises from viewing $G$ as a scheme. In that context, one gets a much wider collection of subgroups, due to the presence of non-smooth subgroup schemes of $G$. At the most extreme end, a subgroup $H$ of $G$ is \emph{infinitesimal} if its only $k$-point is the identity element. The most natural non-trivial example of an infinitesimal subgroup is the first Frobenius kernel $G_1$ of $G$; one may view this as the functor from $k$-algebras to groups such such the $A$-points of $G_1$ applied to a $k$-algebra $A$ is the group $G_1(A)=\{x\in G(A)\mid F(x)=1\}$, where $F$ is the (standard) Frobenius map on $G$. 

Recall that a Lie algebra $\g$ is \emph{restricted} if it is equipped with a $[p]$-map $x\mapsto x^{[p]}$ which is $p$-semilinear in $k$ and satisfies $\ad(x^{[p]})(y)=\ad(x)^p(y)$. A subalgebra $\h\subseteq\g$ is a $p$-subalgebra if it is closed under the $[p]$-map. There is an equivalence between $G_1$ and the Lie algebra $\g$ in the following senses: $\Lie(G_1)=\g$ is a restricted Lie algebra; any finite-dimensional restricted Lie algebra $\mathfrak{k}$ is $\Lie(K)$ for a unique connected height-one group scheme $K$; under this correspondence, any $p$-subalgebra $\h$ of $\g$ maps to a unique subgroup $H$ of $G_1$; the finite-dimensional representation theory of $G_1$ is equivalent to the finite-dimensional restricted representation theory of $\g$. For more on this, see the article of Brion from this volume.

Moving from the smooth subgroups of $G$ to the subalgebras of $\g$ introduces many new and difficult problems. First, the classification of simple Lie algebras in positive characteristic \cite{PS06} is vastly more complicated than that of the algebraic groups, and is only complete when $p>3$. Second, semisimple subalgebras are not the sums of simple Lie algebras, or even closely related to them \cite[\S3.3]{Str04}. Third, it is not in general true that Theorem \ref{thm:bt} has an analogue for $\g$, and indeed maximal non-semisimple subalgebras of $\g$ do not have to be parabolic---where a Lie subalgebra of $\g$ is called \emph{parabolic} if it is the Lie algebra of a parabolic subgroup. Indeed, \cite[p.~149]{Str04} describes (all) irreducible representations of the soluble Heisenberg Lie algebras, most of which have dimensions divisible by $p$.

Circumventing these problems in the case that $G$ is classical is wide open. But at least when $G$ is of exceptional type in good characteristic, these problems have been dealt with in \cite{HSMax}, \cite{Pre17} and \cite{PSMax}. The analogue of the Liebeck--Seitz classification of maximal connected subgroups of $G$ holds, with certain exceptions. For example, for $p\geq 3$ the first Witt algebra $W_1:=\Der(k[X]/X^p)$ is a simple Lie algebra of dimension $p$ and appears between $G$ and its regular $\sl_2$ subalgebra whenever $p=h+1$ where $h$ is the Coxeter number of $G$. There are some maximal semisimple Lie algebras when $G$ has type $E_7$ and $p=5$ or $7$, which have nothing to do with semisimple subgroups of $G$. We also point out that \cite[Corollary~1.4]{PSMax} establishes an exact analogue of the Borel--Tits Theorem for exceptional Lie algebras in good characteristic.

One would like to consider the analogues for Lie algebras of the main problem addressed in this article. The following definition was given in \cite{McN07} and developed in \cite{MR2860266}.

\begin{defn} Let $\g=\Lie(G)$ for $G$ a reductive algebraic group. Then a subalgebra $\h$ of $\g$ is $G$-cr if whenever $\h$ is in a parabolic subalgebra $\p=\Lie(P)$ of $G$, then $\h$ is in a Levi subalgebra $\l=\Lie(L)$ of $\p$, where $P=QL$ is a Levi decomposition of $P$.\end{defn}

An analogue of Theorem \ref{thm:gcr-red} for Lie algebras (building on work in \cite{HS16}) is given in \cite[Theorem~1.3]{ST18}:

\begin{theorem}\label{thm:gcralg}Let $G$ be a connected reductive algebraic group in characteristic $p$ with Lie algebra $\g$. Suppose that $\h$ is a semisimple subalgebra of $\g$ and $p>h$. Then $\h$ is $G$-cr.\end{theorem}

In the case $\h\cong\sl_2$, the theorem interacts surprisingly closely with Kostant's uniqueness result about the embeddings of nilpotent elements into $\sl_2$-subalgebras: it builds on the Jacobson--Morozov theorem \cite{MR49882,MR0007750}, which says that for any complex finite-dimensional semisimple Lie algebra $\g=\Lie(G)$, there is a surjective map \[ \{\text{conjugacy classes of }\sl_2\text{-triples}\}\longrightarrow\{\text{nilpotent orbits in $\g$}\}, \tag{$\ast$}\label{eq:star}\]
where an $\sl_2$-triple is a triple $(e,h,f)\in \g^3$ satisfying $[h,e]=2e$, $[h,f]=-2f$, $[e,f]=h$. The surjective map is induced by sending $(e,h,f)$ to the nilpotent element $e$. So any such $e$ can be embedded into some $\sl_2$-triple. In \cite{Kos59}, Kostant showed that this can be done uniquely up to conjugacy by the centraliser $G_e$ of $e$; i.e.~the map \eqref{eq:star} is actually a bijection. Much work has been done on extending this important result into characteristic $p > 0$. 
We mention some critical contributions. In \cite{Pom80}, Pommerening showed that under the mild restriction that $p$ is a good prime for $G$, one can always find an $\sl_2$-subalgebra containing a given nilpotent element, but this may not be unique; in other words, the map \eqref{eq:star} is still surjective, but not necessarily injective.  In \cite{SpSt70} Springer and Steinberg prove that the uniqueness holds whenever $p\geq 4h-1$ and in his book \cite{Car93}, Carter uses an argument due to Spaltenstein to reduce this bound to $p>3h-3$; both proofs make use of an exponentiation argument. One use of Theorem \ref{thm:gcralg} is to prove the following.

\begin{theorem}\label{jmthm} Let $G$ be a connected reductive group in characteristic $p>2$ with Lie algebra $\g$. Then \eqref{eq:star} is a bijection if and only if $p>h$.\end{theorem}

In fact, \cite{ST18} also considers a map \[\{\text{conjugacy classes of }\sl_2\text{-subalgebras}\}\rightarrow\{\text{nilpotent orbits in $\g$}\}\tag{$\ast\ast$}\label{eq:starstar},\]
and when a bijection exists, realises it in a natural way. The equivalence of bijections \eqref{eq:star} and \eqref{eq:starstar} is easily seen in large enough characteristics by exponentiation, but there are quite a few characteristics where there exists a bijection \eqref{eq:starstar}, but not \eqref{eq:star}.

Further progress on this theme has been made by Goodwin--Pengelly \cite{GP22}, characterising the subvarieties of nilpotent elements where bijections \eqref{eq:star} and \eqref{eq:starstar} hold.

\bibliographystyle{cambridgeauthordate}
\bibliography{LSTbiblio}	

\end{document}